\documentclass{amsart}
\usepackage{hyperref,url}
 \newtheorem{theorem}{Theorem}[section]
 \newtheorem{corollary}[theorem]{Corollary}
 \newtheorem{lemma}[theorem]{Lemma}
 
 \newtheorem{conjecture}[theorem]{Conjecture}
 \theoremstyle{definition}
 
 \theoremstyle{remark}
 \newtheorem{remark}[theorem]{Remark}
 
 \numberwithin{equation}{section}

\newcommand{\C}{\mathbb{C}}
\newcommand{\R}{\mathbb{R}}
\newcommand{\W}{\mathbb{W}}
\newcommand{\Z}{\mathbb{Z}}
\newcommand{\cA}{\mathcal{A}}
\newcommand{\cB}{\mathcal{B}}
\newcommand{\cC}{\mathcal{C}}
\newcommand{\cG}{\mathcal{G}}
\newcommand{\cJ}{\mathcal{J}}
\newcommand{\cK}{\mathcal{K}}
\newcommand{\cL}{\mathcal{L}}
\newcommand{\cM}{\mathcal{M}}
\newcommand{\cS}{\mathcal{S}}
\newcommand{\cZ}{\mathcal{Z}}
\newcommand{\alg}{\operatorname{alg}}
\newcommand{\diag}{\operatorname{diag}}
\newcommand{\im}{\operatorname{Im}}
\newcommand{\Ker}{\operatorname{Ker}}
\newcommand{\eps}{\varepsilon}
\begin{document}

\title[Singular Integral Operators]
 {Singular Integral Operators on Variable Lebesgue Spaces
 with Radial Oscillating Weights}
\author[Karlovich]{Alexei Yu. Karlovich}
\address{
Departamento de Matem\'atica\\
Faculdade de Ci\^encias e Tecnologia\\
Universidade Nova de Lisboa\\
Quinta da Torre\\
2829--516 Caparica\\
Portugal}
\email{oyk@fct.unl.pt}
\thanks{The author is partially supported by the grant FCT/FEDER/POCTI/MAT/59972/2004.}


\dedicatory{To the memory of Igor Borisovich Simonenko (1935--2008)}
\begin{abstract}
We prove a Fredholm criterion for operators in the Banach algebra of singular
integral operators with matrix piecewise continuous coefficients acting on a
variable Lebesgue space with a radial oscillating weight over a logarithmic
Carleson curve. The local spectra of these operators are massive and have a
shape of spiralic horns depending on the value of the variable exponent,
the spirality indices of the curve, and the Matuszewska-Orlicz indices of
the weight at each point. These results extend (partially) the results of
A.~B\"ottcher, Yu.~Karlovich, and V.~Rabinovich for standard Lebesgue spaces
to the case of variable Lebesgue spaces.
\end{abstract}
\maketitle
\section{Introduction}
Let $X$ be a Banach space and $\cB(X)$ be the Banach algebra of all bounded
linear operators on $X$. An operator $A\in\cB(X)$ is said to be $n$-normal
(resp. $d$-normal) if its image $\im A$ is closed in $X$ and the defect number
$n(A;X):=\dim\Ker A$ (resp. $d(A;X):=\dim\Ker A^*$) is finite. An operator $A$
is said to be semi-Fredholm on $X$ if it is $n$-normal or $d$-normal. Finally,
$A$ is said to be Fredholm if it is simultaneously $n$-normal and $d$-normal.
Let $N$ be a positive integer. We denote by $X_N$ the direct sum of $N$ copies of
$X$ with the norm
\[
\|f\|=\|(f_1,\dots,f_N)\|:=(\|f_1\|^2+\dots+\|f_N\|^2)^{1/2}.
\]

Let $\Gamma$ be a Jordan curve, that is, a curve that is homeomorphic to a circle.
We suppose that $\Gamma$ is rectifiable. We equip $\Gamma$ with Lebesgue length
measure $|d\tau|$ and the counter-clockwise orientation. The \textit{Cauchy
singular integral} of $f\in L^1(\Gamma)$ is defined by
\[
(Sf)(t):=\lim_{R\to 0}\frac{1}{\pi i}\int_{\Gamma\setminus\Gamma(t,R)}
\frac{f(\tau)}{\tau-t}d\tau
\quad (t\in\Gamma),
\]
where $\Gamma(t,R):=\{\tau\in\Gamma:|\tau-t|<R\}$ for $R>0$.
David \cite{David84} (see also \cite[Theorem~4.17]{BK97}) proved that the
Cauchy singular integral generates the bounded operator $S$ on the Lebesgue
space $L^p(\Gamma)$, $1<p<\infty$, if and only if $\Gamma$ is a
\textit{Carleson} (\textit{Ahlfors-David regular}) \textit{curve}, that is,
\[
\sup_{t\in\Gamma}\sup_{R>0}\frac{|\Gamma(t,R)|}{R}<\infty,
\]
where $|\Omega|$ denotes the measure of a measurable set $\Omega\subset\Gamma$.
We can write $\tau-t=|\tau-t|e^{i\arg(\tau-t)}$
for $\tau\in\Gamma\setminus\{t\}$,
and the argument can be chosen so that it is continuous on $\Gamma\setminus\{t\}$.
It is known \cite[Theorem~1.10]{BK97} that for an arbitrary Carleson curve the
estimate
\[
\arg(\tau-t)=O(-\log|\tau-t|)\quad (\tau\to t)
\]
holds for every $t\in\Gamma$. One says
that a Carleson curve $\Gamma$ satisfies the \textit{logarithmic whirl condition}
at $t\in\Gamma$ if
\begin{equation}\label{eq:spiralic}
\arg(\tau-t)=-\delta(t)\log|\tau-t|+O(1)\quad (\tau\to t)
\end{equation}
with some $\delta(t)\in\R$. Notice that all piecewise smooth curves satisfy this
condition at each point and, moreover, $\delta(t)\equiv 0$. For more information
along these lines, see \cite{BK95}, \cite[Chap.~1]{BK97}, \cite{BK01}.

A measurable function $w:\Gamma\to[0,\infty]$ is referred to as a \textit{weight
function} or simply a \textit{weight} if $0<w(\tau)<\infty$ for almost all
$\tau\in\Gamma$. Suppose $p:\Gamma\to(1,\infty)$ is a continuous function.
Denote by $L^{p(\cdot)}(\Gamma,w)$ the set of all measurable complex-valued
functions $f$ on $\Gamma$ such that
\[
\int_\Gamma |f(\tau)w(\tau)/\lambda|^{p(\tau)}|d\tau|<\infty
\]
for some $\lambda=\lambda(f)>0$. This set becomes a Banach space when equipped
with the Luxemburg-Nakano norm
\[
\|f\|_{p(\cdot),w}:=\inf\left\{\lambda>0:
\int_\Gamma |f(\tau)w(\tau)/\lambda|^{p(\tau)}|d\tau|\le 1\right\}.
\]
If $p$ is constant, then $L^{p(\cdot)}(\Gamma,w)$ is nothing else than
the weighted Lebesgue space. Therefore, it is natural to refer to
$L^{p(\cdot)}(\Gamma,w)$ as a \textit{weighted generalized Lebesgue space
with variable exponent} or simply as a \textit{weighted variable Lebesgue
space}. This is a special case of Musielak-Orlicz spaces \cite{Musielak83}
(see also \cite{KR91}). Nakano \cite{Nakano50} considered these spaces
(without weights) as examples of so-called modular spaces, and sometimes
the spaces $L^{p(\cdot)}(\Gamma,w)$ are referred to as weighted Nakano
spaces. Since $\Gamma$ is compact, we have
\[
1<\min_{t\in\Gamma}p(t),
\quad
\max_{t\in\Gamma}p(t)<\infty.
\]
Therefore, if $w\in L^{p(\cdot)}(\Gamma)$ and $1/w\in L^{q(\cdot)}(\Gamma)$,
where
\[
q(t):=p(t)/(p(t)-1)
\]
is the conjugate exponent for $p$, then $L^{p(\cdot)}(\Gamma,w)$ is reflexive
and its Banach dual can be identified with $L^{q(\cdot)}(\Gamma,1/w)$
(see e.g. \cite[Section~13]{Musielak83}, \cite[Corollary~2.7]{KR91}
and also \cite[Section~2]{Karlovich03}).

Following \cite[Section~2.3]{KSS07}, denote by $W$ the class of all
continuous functions $\varrho: [0,|\Gamma|]\to[0,\infty)$ such that
$\varrho(0)=0$, $\varrho(x)>0$ if $0<x\le |\Gamma|$, and $\varrho$ is
almost increasing, that is, there is a universal constant $C>0$ such
that $\varrho(x)\le C\varrho(y)$ whenever $x\le y$. Further, let $\W$ be
the set of all functions $\varrho:[0,|\Gamma|]\to[0,\infty]$ such that
$x^\alpha\varrho(x)\in W$ and $x^\beta/\varrho(x)\in W$ for some
$\alpha,\beta\in\R$. Clearly, the functions $\varrho(x)=x^\gamma$
belong to $\W$ for all $\gamma\in\R$. For $\varrho\in\W$, put
\[
\Phi_\varrho^0(x):=\limsup_{y\to 0}\frac{\varrho(xy)}{\varrho(y)},
\quad
x\in(0,\infty).
\]
Since $\varrho\in\W$, one can show (see Subsection~\ref{sect:MO})
that the limits
\begin{equation}\label{eq:MO-definition}
m(\varrho):=\lim_{x\to 0}\frac{\log\Phi_\varrho^0(x)}{\log x},
\quad
M(\varrho):=\lim_{x\to\infty}\frac{\log\Phi_\varrho^0(x)}{\log x}
\end{equation}
exist and $-\infty<m(\varrho)\le M(\varrho)<+\infty$. These numbers were
defined under some extra assumptions on $\varrho$ by Matuszewska
and Orlicz \cite{MO60,MO65} (see also \cite{M85} and \cite[Chapter~11]{M89}).
We refer to $m(\varrho)$ (resp. $M(\varrho)$) as the \textit{lower}
(resp. \textit{upper}) \textit{Matuszewska-Orlicz index} of $\varrho$.
For $\varrho(x)=x^\gamma$ one has $m(\varrho)=M(\varrho)=\gamma$.
Examples of functions $\varrho\in\W$ with $m(\varrho)<M(\varrho)$ can be
found, for instance, in \cite{AK89}, \cite[p.~93]{M89}, \cite[Section~2]{NSamko06}.

Fix pairwise distinct points $t_1,\dots,t_n\in\Gamma$ and functions
$w_1,\dots,w_n\in\W$. Consider the following weight
\begin{equation}\label{eq:weight}
w(t):=\prod_{k=1}^n w_k(|t-t_k|),
\quad
t\in\Gamma.
\end{equation}
Each function $w_k(|t-t_k|)$ is a radial oscillating weight. This is a
natural generalization of so-called Khvedelidze weights
$w(t)=\prod_{k=1}^n|t-t_k|^{\lambda_k}$, where $\lambda_k\in\R$
(see, e.g., \cite[Section~2.2]{BK97}, \cite[Section~1.4]{GK92},
\cite{K56}, \cite{KPS06}).
\begin{theorem}\label{th:KSS}
Suppose $\Gamma$ is a Carleson Jordan curve and $p:\Gamma\to(1,\infty)$
is a continuous function satisfying
\begin{equation}\label{eq:Dini-Lipschitz}
|p(\tau)-p(t)|\le -A_\Gamma/\log|\tau-t|
\quad\mbox{whenever}\quad
|\tau-t|\le1/2,
\end{equation}
where $A_\Gamma$ is a positive constant depending only on $\Gamma$. Let
$w_1,\dots,w_n\in\W$ and the weight $w$ be given by \eqref{eq:weight}.
The Cauchy singular integral operator $S$ is bounded on $L^{p(\cdot)}(\Gamma,w)$
if and only if
\begin{equation}\label{eq:KSS-condition}
0<1/p(t_k)+m(w_k),\quad 1/p(t_k)+M(w_k)<1
\quad\mbox{for all}\quad
k\in\{1,\dots,n\}.
\end{equation}
\end{theorem}
For $w(t)=\prod_{k=1}^n|t-t_k|^{\lambda_k}$, \eqref{eq:KSS-condition}
reads as $0<1/p(t_k)+\lambda_k<1$ for all $k\in\{1,\dots,n\}$ and for these
weights Theorem~\ref{th:KSS} is obtained in \cite{KPS06}. The sufficiency
portion of Theorem~\ref{th:KSS} in the form stated above is obtained
by V.~Kokilashvili, N.~Samko, and S.~Samko \cite[Theorem~4.3]{KSS07}.
The necessity portion is a new result. It will be proved in Section~\ref{sec:4.1}.

We define by $PC(\Gamma)$ as the set of all $a\in L^\infty(\Gamma)$ for which
the one-sided limits
\[
a(t\pm 0):=\lim_{\tau\to t\pm 0}a(\tau)
\]
exist and finite at each point $t\in\Gamma$; here $\tau\to t-0$ means that $\tau$
approaches $t$ following the orientation of $\Gamma$, while $\tau\to t+0$
means that $\tau$ goes to $t$ in the opposite direction. Functions in $PC(\Gamma)$
are called \textit{piecewise continuous} functions.

The operator $S$ is defined on $L_N^{p(\cdot)}(\Gamma,w)$ elementwise.
We let stand $PC_{N\times N}(\Gamma)$ for the algebra of all $N\times N$
matrix functions with entries in $PC(\Gamma)$. Writing the elements of
$L_N^{p(\cdot)}(\Gamma,w)$ as columns, we can define the multiplication
operator $aI$ for $a\in PC_{N\times N}(\Gamma)$ as multiplication by the
matrix function $a$. Let $\alg(S,PC,L_N^{p(\cdot)}(\Gamma,w))$ denote the
smallest closed subalgebra of $\cB(L_N^{p(\cdot)}(\Gamma,w))$ containing
the operator $S$ and the set $\{aI:a\in PC_{N\times N}(\Gamma)\}$.

For the case of constant $p\in(1,\infty)$, Khvedelidze weights $w$, and
piecewise Lyapunov curves $\Gamma$, the algebra $\alg(S,PC,L_N^p(\Gamma,w))$
was well understood in 1970--1980's (see e.g. \cite{GK71,GK92,Krupnik87,RS90,SM86}).
In the earlier 1990's, Spitkovsky \cite{Spitkovsky92} discovered that local
spectra of singular integral operators on Lebesgue spaces with Muckenhoupt
weights on smooth curves have a shape of horns bounded by two circular arcs
depending on the so-called indices of powerlikeness of the weight. In the
middle of 1990's, B\"ottcher, Yu.~Karlovich, and Rabinovich further observed
that these horns metamorphose to spiralic horns bounded by logarithmic double
spirals \cite{BK95,BKR96,Rabinovich96} if one passes from nice curves to
Carleson curves satisfying \eqref{eq:spiralic}. These spiralic horns become
even more interesting creatures (so-called ``general leaves") in case of
arbitrary Carleson curves \cite{BK97,BK01}. The shape of a general leaf depends
on the spirality indices of the curve and the indices of powerlikeness of the
weight. The Fredholm theory for the algebra $\alg(S,PC,L_N^p(\Gamma,w))$
under the most general conditions on the curve $\Gamma$ and the weight $w$
is constructed by B\"ottcher and Yu.~Karlovich and is presented in the monograph
\cite{BK97} (although, we advise to start the study of this theory from the nice
survey \cite{BK01}).

Partially, the results of \cite{BK97} are extended to Orlicz spaces
\cite{Karlovich96}, and further to rearrangement-invariant spaces \cite{Karlovich98},
and rearrangement-invariant spaces with Muckenhoupt weights \cite{Karlovich02}.
The results of these papers resemble those of \cite{BK97} with one important
difference: the number $1/p$ is replaced by the Boyd indices $\alpha_X,\beta_X$
of the rearrangement-invariant space. Since the Boyd indices may be different,
one more factor (a general space) leads to massive local spectra (for standard
Lebesgue spaces these factors were general weights and general curves). Local
spectra of singular integral operators can be massive also on weighted H\"older
spaces on Carleson curves \cite{RSS06}.

The study of singular integral operators (SIOs) with discontinuous coefficients
on weighted variable Lebesgue spaces on sufficiently nice curves was started
in \cite{Karlovich03,Karlovich06,KPS05,KS03}. A Fredholm criterion for an
arbitrary operator in $\alg(S,PC,L_N^{p(\cdot)}(\Gamma,w))$ for Carleson curves
satisfying \eqref{eq:spiralic} and Khvedelidze weights is proved in
\cite{Karlovich05}. Local spectra in that paper have a shape of logarithmic
double spirals and cannot be massive. Under the same assumptions on the curve
and the weight, it is proved in \cite{Karlovich07} that every semi-Fredholm
operator in the algebra $\alg(S,PC,L_N^{p(\cdot)}(\Gamma,w))$ is Fredholm.

In this paper we extend the results of \cite{Karlovich05,Karlovich06,Karlovich07} to the
case of Carleson curves satisfying \eqref{eq:spiralic} and weights of the form
\eqref{eq:weight} satisfying the conditions of Theorem~\ref{th:KSS}.
Local spectra of SIOs
corresponding to the points $t\in\Gamma\setminus\{t_1,\dots,t_n\}$
are logarithmic double spirals depending on $\delta(t)$ and $1/p(t)$.
These spirals blow up to spiralic horns at $t=t_k$ if $k\in\{1,\dots,n\}$
and $m(w_k)<M(w_k)$. These spiralic horns are bounded by two logarithmic double
spirals depending on $\delta(t_k)$ and on $1/p(t_k)+m(w_k)$ and $1/p(t_k)+M(w_k)$,
respectively. Up to our knowledge, this paper is the first work, where
massive local spectra of singular integral operators appear in the setting of
weighted variable Lebesgue spaces. These results resemble those of
\cite{BK95,BKR96} for weighted standard Lebesgue spaces, although the weights
considered in \cite{BK95,BKR96} are more general than in the present paper.

The paper is organized as follows. In Section~\ref{sect:2} we collect all
necessary information on indices of submultiplicative functions associated
with curves and weights. We prove that the indices of powerlikeness (see
\cite[Chap.~3]{BK97}) of the weight \eqref{eq:weight} at $t_k$ coincide
with the Matuszewska-Orlicz indices of $w_k$. Section~\ref{sect:3} contains
standard results on singular integral operators with $L^\infty$ coefficients:
the necessary condition for Fredholmness, the local principle of Simonenko
type, and the theorem on a Wiener-Hopf factorization. In Section~\ref{sect:4},
we prove a Fredholm criterion for a singular integral operator $aP+Q$, where
$a\in PC(\Gamma)$ and $P:=(I+S)/2$, $Q:=(I-S)/2$.
In a sense, the main result of Section~\ref{sect:4} is the heart of the paper.
It follows from a more general necessary condition for the Fredholmness of $aP+Q$
(see \cite[Theorem~8.1]{Karlovich03}) and a sufficient condition for
the Fredholmness of $aP+Q$, whose proof is based on the local principle of
Simonenko type and the Wiener-Hopf factorization of a local representative
of $a$.
In fact, \cite[Therem~8.1]{Karlovich03} together with the results of
Section~\ref{sect:2} imply the necessity of the conditions \eqref{eq:KSS-condition}
for the boundedness of the operator $S$ on $L^{p(\cdot)}(\Gamma,w)$.
Section~\ref{sect:5} contains the Allan-Douglas local principle
and the two projections theorem. These results are our main tools in the
construction of a symbols calculus for the algebra $\alg(S,PC,L_N^{p(\cdot)}(\Gamma,w))$
on the basis of the main result of Section~\ref{sect:4}. In Section~\ref{sect:6},
following the well known scheme (see e.g. \cite{BK95}, \cite[Chap.~8]{BK97}
and also \cite{Karlovich96,Karlovich02,Karlovich05,Karlovich06}),
we prove a Fredholm criterion for an arbitrary operator
$A\in\alg(S,PC,L_N^{p(\cdot)}(\Gamma,w))$. Finally, we collect some
remarks on index  formulas and semi-Fredholm operators in the algebra
$\alg(S,PC,L_N^{p(\cdot)}(\Gamma,w))$ and discuss open problems.
\section{Submultiplicative functions and their indices}
\label{sect:2}
\subsection{Submultiplicative functions}
Following \cite[Section~1.4]{BK97}, we say a function
$\Phi:(0,\infty)\to(0,\infty]$ is \textit{regular} if it is bounded in an
open neighborhood of $1$. A function $\Phi:(0,\infty)\to(0,\infty]$ is said
to be \textit{submultiplicative} if
\[
\Phi(xy)\le\Phi(x)\Phi(y)
\quad\mbox{for all}\quad x,y\in(0,\infty).
\]
It is easy to show that if $\Phi$ is regular and submultiplicative, then
$\Phi$ is bounded away from zero in some open neighborhood of $1$.
Moreover, in this case $\Phi(x)$ is finite for all $x\in(0,\infty)$.
Given a regular and submultiplicative function $\Phi:(0,\infty)\to(0,\infty)$,
one defines
\[
\alpha(\Phi):=\sup_{x\in(0,1)}\frac{\log\Phi(x)}{\log x},
\quad
\beta(\Phi):=\inf_{x\in(1,\infty)}\frac{\log\Phi(x)}{\log x}.
\]
Clearly, $-\infty<\alpha(\Phi)$ and $\beta(\Phi)<\infty$.
\begin{theorem}[see \cite{BK97}, Theorem~1.13 or \cite{KPS82}, Chap.~2, Theorem~1.3]
\label{th:submult}
If a function $\Phi:(0,\infty)\to(0,\infty)$ is regular and submultiplicative, then
\[
\alpha(\Phi)=\lim_{x\to 0}\frac{\log\Phi(x)}{\log x},
\quad
\beta(\Phi)=\lim_{x\to\infty}\frac{\log\Phi(x)}{\log x}
\]
and $-\infty<\alpha(\Phi)\le\beta(\Phi)<+\infty$.
\end{theorem}
The quantities $\alpha(\Phi)$ and $\beta(\Phi)$ are called the \textit{lower}
and  \textit{upper indices of the regular and submultiplicative function}
$\Phi$, respectively.
\subsection{Matuszewska-Orlicz indices}
\label{sect:MO}
\begin{lemma}\label{le:Phi-submult}
Let $A>0$ and $\varrho:(0,A]\to(0,\infty)$ be a continuous function. Then for
every $B\in(0,A]$ the function
\begin{equation}\label{eq:Phi-submult}
\Phi_{\varrho,B}(x):=\left\{
\begin{array}{lll}
\displaystyle
\sup_{0<y\le B}\frac{\varrho(xy)}{\varrho(y)} &\mbox{if}& x\in(0,1],
\\[3mm]
\displaystyle
\sup_{0<y\le B}\frac{\varrho(y)}{\varrho(x^{-1}y)} &\mbox{if}& x\in(1,\infty)
\end{array}
\right.
\end{equation}
is submultiplicative.
\end{lemma}
This statement is proved similarly to \cite[Lemma~1.15]{BK97}.
\begin{lemma}\label{le:Phi-Phi-0}
Let $A>0$ and $\varrho:(0,A]\to(0,\infty)$ be a continuous function. If for
some $B\in(0,A]$ the function $\Phi_{\varrho,B}$ given by \eqref{eq:Phi-submult}
is regular, then
\[
\Phi_\varrho^0(x):=\limsup_{y\to 0}\frac{\varrho(xy)}{\varrho(y)}
\]
is regular and submultiplicative and both functions have the same lower and
upper indices
\[
\alpha(\Phi_{\varrho,B})=\alpha(\Phi_\varrho^0)=:m(\varrho),
\quad
\beta(\Phi_{\varrho,B})=\beta(\Phi_\varrho^0)=:M(\varrho).
\]
\end{lemma}
This statement is proved by analogy with \cite[Lemma~2(a)]{Boyd71}
(see also \cite[Theorem~8.18]{BS88} and \cite[Lemma~1.16]{BK97}).
\begin{lemma}\label{le:Phi-regular}
If $\varrho\in\W$, then for every $B\in(0,|\Gamma|]$ the function $\Phi_{\varrho,B}$
is regular.
\end{lemma}
\begin{proof}
Since $\varrho\in\W$, there exist $a,b\in\R$ such that the functions
$x^a\varrho(x)$ and $x^b/\varrho(x)$ are almost increasing on $(0,|\Gamma|]$,
that is, there exist positive constants $c_a,c_b$ such that
$x^a\varrho(x)\le c_ay^a\varrho(y)$ and $x^b/\varrho(x)\le c_by^b/\varrho(y)$
whenever $x\le y$ and $x,y\in(0,|\Gamma|]$. Suppose $x\in(0,1]$ and
$y\in(0,|\Gamma|]$. Then
\[
\frac{\varrho(xy)}{\varrho(y)}=\frac{(xy)^a\varrho(xy)}{(xy)^a\varrho(y)}
\le
\frac{c_ay^a\varrho(y)}{(xy)^a\varrho(y)}=\frac{c_a}{x^a}.
\]
Hence
\[
\Phi_{\varrho,B}(x)\le\Phi_{\varrho,|\Gamma|}(x)\le c_ax^{-a},
\quad
x\in(0,1].
\]
Similarly,
\[
\Phi_{\varrho,B}(x)\le\Phi_{\varrho,|\Gamma|}(x)\le c_bx^b,
\quad
x\in(1,\infty).
\]
Thus, $\Phi_{\varrho,B}$ is regular for every $B\in(0,|\Gamma|]$.
\end{proof}
From the above results and Theorem~\ref{th:submult} we conclude that if
$\varrho\in\W$, then its Matuszewska-Orlicz indices are well defined by
\eqref{eq:MO-definition}.
\subsection{Spirality indices of Carleson curves}
Fix $t\in\Gamma$ and put
\[
d_t:=\max_{\tau\in\Gamma}|\tau-t|.
\]
Suppose $\psi:\Gamma\setminus\{t\}\to(0,\infty)$ is a continuous function
and consider
\[
F_{\psi,t}(R_1,R_2):=\max_{\tau\in\Gamma,|\tau-t|=R_1}\psi(\tau)\Big/
\min_{\tau\in\Gamma,|\tau-t|=R_2}\psi(\tau),
\quad
R_1,R_2\in(0,d_t].
\]
By \cite[Lemma~1.15]{BK97}, the function
\[
(W_t\psi)(x):=\left\{\begin{array}{lll}
\displaystyle
\sup_{0<R\le d_t}F_{\psi,t}(xR,R) & \mbox{if} &x\in(0,1],\\
\displaystyle
\sup_{0<R\le d_t}F_{\psi,t}(R,x^{-1}R) & \mbox{if} & x\in(1,\infty)
\end{array}
\right.
\]
is submultiplicative. For $t\in\Gamma$, we have,
\[
\tau- t=|\tau-t|e^{i\arg(\tau-t)},
\quad
\tau\in\Gamma\setminus\{t\},
\]
and the argument $\arg(\tau-t)$ may be chosen to be continuous on
$\Gamma\setminus\{t\}$. Consider
\[
\eta_t(\tau):=e^{-\arg(\tau-t)}.
\]
\begin{lemma}[see \cite{BK97}, Theorem~1.18]
\label{le:spirality-existence}
If $\Gamma$ is a Carleson Jordan curve, then for every $t\in\Gamma$
the function $W_t\eta_t$ is regular and submultiplicative. If, in addition,
it satisfies \eqref{eq:spiralic} at some point $t\in\Gamma$, then
\[
\alpha(W_t\eta_t)=\beta(W_t\eta_t)=\delta(t).
\]
\end{lemma}
The numbers $\alpha(W_t\eta_t)$ and $\beta(W_t\eta_t)$ are called the
\textit{lower} and \textit{upper spirality indices} of $\Gamma$ at $t$,
respectively. If $\Gamma$ is a piecewise smooth curve, then
$\alpha(W_t\eta_t)=\beta(W_t\eta_t)=0$
for all $t\in\Gamma$ (non-spiral curves). Carleson curves satisfying
\eqref{eq:spiralic} behave like perturbed logarithmic spirals. Examples
of Carleson curves which have arbitrary prescribed and distinct spirality
indices are given in \cite[Section~1.6]{BK97}.
\subsection{Indices of powerlikeness of continuous nonvanishing weights}
Let $\psi$ be a weight on $\Gamma$ such that $\log\psi\in L^1(\Gamma(t,R))$ for
every $R\in(0,d_t]$. Put
\[
H_{\psi,t}(R_1,R_2):=
\frac{\displaystyle\exp\left(\frac{1}{|\Gamma(t,R_1)|}
\int_{\Gamma(t,R_1)}\log\psi(\tau)|d\tau|\right)}
{\displaystyle\exp\left(\frac{1}{|\Gamma(t,R_2)|}
\int_{\Gamma(t,R_2)}\log\psi(\tau)|d\tau|\right)},
\quad R_1,R_2\in(0,d_t].
\]
Consider the function
\[
(V_t^0\psi)(x) := \limsup_{R\to 0} H_{\psi,t}(xR,R),
\quad x\in(0,\infty).
\]
\begin{lemma}
\label{le:BK97-Lemma3.16}
Let $\Gamma$ be a Carleson Jordan curve. Suppose $\psi:\Gamma\to[0,\infty]$
is a weight which is continuous and nonvanishing on $\Gamma\setminus\{t\}$.
If $W_t\psi$ is regular, then $V_t^0\psi$ is regular and submultiplicative
and
\[
\alpha(W_t\psi)=\alpha(V_t^0\psi),
\quad
\beta(W_t\psi)=\beta(V_t^0\psi).
\]
\end{lemma}
This statement follows from \cite[Theorem~3.3(c) and Lemma~3.16]{BK97}.

The numbers $\alpha(V_t^0\psi)$ and $\beta(V_t^0\psi)$ are called the
\textit{lower} and \textit{upper indices of powerlikeness} of $\psi$ at
$t\in\Gamma$, respectively. This terminology can be explained by
the simple fact that for the power weight $\psi(\tau):=|\tau-t|^\lambda$
its indices of powerlikeness coincide and are equal to $\lambda$.
Examples of Muckenhoupt weights with distinct indices of powerlikeness
are given in \cite[Proposition 3.25 and Examples 3.26--3.27]{BK97}.
\subsection{Indices of powerlikeness of radial oscillating weights}
\begin{lemma}\label{le:indices-Phi-Vt}
Let $\Gamma$ be a Carleson Jordan curve and $\varrho\in\W$. Suppose
\[
\psi_t(\tau):=\varrho(|\tau-t|)\mbox{ for }\tau\in\Gamma\setminus\{t\}.
\]
Then the functions $W_t\psi_t$ and $V_t^0\psi_t$ are regular and
submultiplicative and
\[
m(\varrho)=\alpha(W_t\psi_t)=\alpha(V_t^0\psi_t),
\quad
M(\varrho)=\beta(W_t\psi_t)=\beta(V_t^0\psi_t).
\]
\end{lemma}
\begin{proof}
By Lemmas~\ref{le:Phi-submult}--\ref{le:Phi-regular}, the functions
$\Phi_{\varrho,d_t}$ and $\Phi_\varrho^0$ are regular and submultiplicative
and
\begin{equation}\label{eq:indices-Phi-Vt-1}
\alpha(\Phi_{\varrho,d_t})=\alpha(\Phi_\varrho^0)=:m(\varrho),
\quad
\beta(\Phi_{\varrho,d_t})=\beta(\Phi_\varrho^0)=:M(\varrho).
\end{equation}
If $x\in(0,1]$ and $0<R\le d_t$, then
\begin{equation}\label{eq:indices-Phi-Vt-2}
F_{\psi_t,y}(xR,R)
=
\frac{\displaystyle\max_{\tau\in\Gamma,|\tau-t|=xR}\varrho(|\tau-t|)}
{\displaystyle\min_{\tau\in\Gamma,|\tau-t|=R}\varrho(|\tau-t|)}
=
\frac{\varrho(xR)}{\varrho(R)};
\end{equation}
if $x\in[1,\infty)$ and $0<R\le d_t$, then
\begin{equation}\label{eq:indices-Phi-Vt-3}
F_{\psi_t,t}(R,x^{-1}R)
=
\frac{\displaystyle\max_{\tau\in\Gamma,|\tau-t|=R}\varrho(|\tau-t|)}
{\displaystyle\min_{\tau\in\Gamma,|\tau-t|=x^{-1}R}\varrho(|\tau-t|)}
=
\frac{\varrho(R)}{\varrho(x^{-1}R)}.
\end{equation}
From \eqref{eq:indices-Phi-Vt-2} and \eqref{eq:indices-Phi-Vt-3} it follows
that
\[
(W_t\psi_t)(x)=\Phi_{\varrho,d_t}(x),
\quad
x\in(0,\infty).
\]
Hence $W_t\psi_t$ is regular because $\Phi_{\varrho,d_t}$ is so and
\begin{equation}\label{eq:indices-Phi-Vt-4}
\alpha(W_t\psi_t)=\alpha(\Phi_{\varrho,d_t}),
\quad
\beta(W_t\psi_t)=\beta(\Phi_{\varrho,d_t}).
\end{equation}
Combining \eqref{eq:indices-Phi-Vt-1}, \eqref{eq:indices-Phi-Vt-4} and
Lemma~\ref{le:BK97-Lemma3.16}, we arrive at the desired statement.
\end{proof}
\begin{theorem}\label{th:radial-MO}
Suppose $\Gamma$ is a Carleson Jordan curve. If $w_1,\dots,w_n\in\W$
and $w(\tau)=\prod_{k=1}^n w_k(|\tau-t_k|)$, then for every $t\in\Gamma$
the function $V_t^0w$ is regular and submultiplicative and
\[
\begin{array}{llll}
\alpha(V_{t_k}^0w)=m(w_k), & \beta(V_{t_k}^0w)=M(w_k) &\mbox{for}  &k\in\{1,\dots,n\},
\\[2mm]
\alpha(V_t^0w)=0, & \beta(V_t^0w)=0 &  \mbox{for}   &t\in\Gamma\setminus\{t_1,\dots,t_n\}.
\end{array}
\]
\end{theorem}
\begin{proof}
If $t\notin\{t_1,\dots,t_n\}$, then there exists $D_t>0$ such that the portion
$\Gamma(t,D_t)$ does not contain any point $t_1,\dots,t_n$. Since the weight
$w$ may vanish or go to infinity only at $t_1,\dots,t_n$ and $w_k(|\tau-t_k|)$
are continuous on $\Gamma\setminus\{t_k\}$, we can conclude that there exist
constants $c_t, C_t\in(0,\infty)$ such that $c_t\le w(\tau)\le C_t$ for all
$\tau\in\Gamma(t,D_t)$. Assume that $x\in(0,\infty)$ and $xR,R\in(0,D_t)$.
Then
\[
e^{c_t-C_t}\le H_{w,t}(xR,R)\le e^{C_t-c_t}.
\]
Hence
\[
e^{c_t-C_t}\le (V_t^0w)(x)\le e^{C_t-c_t},
\quad
x\in(0,\infty),
\]
which implies that $V_t^0w$ is regular. By \cite[Theorem~3.3(c)]{BK97},
the function $V_t^0w$ is submultiplicative. It is easy to see that
$\alpha(V_t^0w)=\beta(V_t^0w)=0$.

If $k\in\{1,\dots,n\}$, then there exists $D_k>0$ such that the portion
$\Gamma(t_k,D_k)$ does not contain any point of $\{t_1,\dots,t_n\}\setminus\{t_k\}$.
As before, there exists positive constants $c_k,C_k$ such that
\[
c_k\le \frac{w(\tau)}{w_k(|\tau-t_k|)}\le C_k
\quad\mbox{for all}\quad \tau\in\Gamma(t_k,D_k).
\]
If $x\in(0,\infty)$ and $xR,R\in(0,D_k)$, then taking into account that
\[
H_{w,t_k}(xR,R)=
H_{w_k(|\tau-t_k|),t_k}(xR,R)\cdot
H_{w/w_k(|\tau-t_k|),t_k}(xR,R),
\]
we get
\[
e^{c_k-C_k}H_{w_k(|\tau-t_k|),t_k}(xR,R)
\le
H_{w,t_k}(xR,R)
\le
e^{C_k-c_k}H_{w_k(|\tau-t_k|),t_k}(xR,R).
\]
Therefore, for all $x\in(0,\infty)$,
\begin{equation}\label{eq:radial-MO-1}
e^{c_k-C_k}\big(V_{t_k}^0w_k(|\tau-t_k|)\big)(x)
\le
(V_{t_k}^0w)(x)
\le
e^{C_k-c_k}\big(V_{t_k}^0w_k(|\tau-t_k|)\big)(x).
\end{equation}
By Lemma~\ref{le:indices-Phi-Vt}, the function $V_{t_k}^0w_k(|\tau-t_k|)$
is regular and
\begin{equation}\label{eq:radial-MO-2}
m(w_k)=\alpha\big(V_{t_k}^0w_k(|\tau-t_k|)\big),
\quad
M(w_k)=\beta\big(V_{t_k}^0w_k(|\tau-t_k|)\big).
\end{equation}
Since the function $V_{t_k}^0w_k(|\tau-t_k|)$ is regular, from
\eqref{eq:radial-MO-1} it follows that $V_{t_k}^0w$ is also regular
and
\begin{equation}\label{eq:radial-MO-3}
\alpha\big(V_{t_k}^0w_k(|\tau-t_k|)\big)=\alpha(V_{t_k}^0w),
\quad
\beta\big(V_{t_k}^0w_k(|\tau-t_k|)\big)=\beta(V_{t_k}^0w).
\end{equation}
Combining \eqref{eq:radial-MO-2} and \eqref{eq:radial-MO-3},
we finally arrive at the equalities $\alpha(V_{t_k}^0w)=m(w_k)$ and
$\beta(V_{t_k}^0w)=M(w_k)$.
\end{proof}
\subsection{Indicator functions}
Let $\Gamma$ be a rectifiable Jordan curve and $p:\Gamma\to(1,\infty)$ be
a continuous function. We say that a weight $w:\Gamma\to[0,\infty]$ belongs
to $A_{p(\cdot)}(\Gamma)$ if
\[
\sup_{t\in\Gamma}\sup_{R>0}\frac{1}{R}
\|w\chi_{\Gamma(t,R)}\|_{p(\cdot)}
\|w^{-1}\chi_{\Gamma(t,R)}\|_{q(\cdot)}<\infty.
\]
If $p=const\in(1,\infty)$, then
this class coincides with the well known Muckenhoupt class. From the H\"older
inequality for $L^{p(\cdot)}(\Gamma)$ (see e.g. \cite[Theorem~2.1]{KR91}) it
follows that if $w\in A_{p(\cdot)}(\Gamma)$, then $\Gamma$ is a Carleson curve.
\begin{theorem}\label{th:Ap}
Let $\Gamma$ be a rectifiable Jordan curve and let $p:\Gamma\to(1,\infty)$
be a continuous function. If $w:\Gamma\to[0,\infty]$ is an arbitrary weight
such that the operator $S$ is bounded on $L^{p(\cdot)}(\Gamma,w)$, then
$w\in A_{p(\cdot)}(\Gamma)$.
\end{theorem}
This statement can be proved by analogy with \cite[Theorem~4.3]{Karlovich96},
\cite[Theorem~3.2]{Karlovich98} (see also \cite[Theorem~4.8]{BK97}).
If $p=const\in(1,\infty)$, then $w\in A_p(\Gamma)$ is also sufficient
for the boundedness of $S$ on the weighted Lebesgue space $L^p(\Gamma,w)$
(see e.g. \cite[Theorem~4.15]{BK97}).
\begin{lemma}\label{le:indicator-existence}
Let $\Gamma$ be a rectifiable Jordan curve, let $p:\Gamma\to(1,\infty)$
be a continuous function satisfying \eqref{eq:Dini-Lipschitz}, and let
$w:\Gamma\to[0,\infty]$ be an arbitrary weight. Suppose
$w\in A_{p(\cdot)}(\Gamma)$. Then
\begin{enumerate}
\item[{\rm (a)}]
for every $t\in\Gamma$ and $x\in\R$,
the function $V_t^0(\eta_t^xw)$ is regular and submultiplicative;

\item[{\rm (b)}]
for every $t\in\Gamma$,
\[
0\le 1/p(t)+\alpha(V_t^0w),
\quad
1/p(t)+\beta(V_t^0w)\le 1.
\]
\end{enumerate}
\end{lemma}
Part (a) follows from \cite[Lemmas 4.5, 5.8, and 5.13]{Karlovich03}.
Part (b) is the consequence of \cite[Lemma~4.9 and Theorem~5.9]{Karlovich03}.

Under the conditions of Lemma~\ref{le:indicator-existence}, the \textit{indicator
functions} of the pair $(\Gamma,w)$ at every $t\in\Gamma$ are well defined by
\[
\alpha_t(x):=\alpha\big(V_t^0(\eta_t^xw)\big),
\quad
\beta_t(x):=\beta\big(V_t^0(\eta_t^xw)\big)
\]
because of Theorem~\ref{th:submult}.
\begin{lemma}\label{le:indicator-calculation}
Let $\Gamma$ be a Carleson Jordan curve satisfying the logarithmic whirl
condition \eqref{eq:spiralic} at every $t\in\Gamma$, let $p:\Gamma\to(1,\infty)$
be a continuous function satisfying \eqref{eq:Dini-Lipschitz}, and let
$w:\Gamma\to[0,\infty]$ be an arbitrary weight. If $w\in A_{p(\cdot)}(\Gamma)$,
then for every $t\in\Gamma$ and $x\in\R$,
\[
\alpha_t(x)=\alpha(V_t^0w)+\delta(t)x,
\quad
\beta_t(x)=\beta(V_t^0w)+\delta(t)x.
\]
\end{lemma}
This result follows from Lemma~\ref{le:spirality-existence} and
\cite[Lemma~5.8, Corollary~5.16(b)]{Karlovich03}.
\section{Singular integral operators with $L^\infty$ coefficients}
\label{sect:3}
\subsection{General necessary condition for Fredholmness}
In this section we will suppose that $\Gamma$ is a Carleson Jordan curve,
$p:\Gamma\to(1,\infty)$ is a continuous function, and $w:\Gamma\to[0,\infty]$
is an arbitrary weight (not necessarily of the form \eqref{eq:weight})
such that $S$ is bounded on $L^{p(\cdot)}(\Gamma,w)$. Under these assumptions,
\[
P:=(I+S)/2,
\quad
Q:=(I-S)/2
\]
are bounded projections on $L^{p(\cdot)}(\Gamma,w)$
(see \cite[Lemma~6.4]{Karlovich03}). The operators of the form $aP+Q$, where
$a \in L^\infty(\Gamma)$, are called \textit{singular integral operators}
(SIOs).
\begin{theorem}\label{th:necessity}
Suppose $a\in L^\infty(\Gamma)$. If $aP+Q$ is Fredholm on $L^{p(\cdot)}(\Gamma,w)$,
then $a^{-1}\in L^\infty(\Gamma)$.
\end{theorem}
This result follows from \cite[Theorem~6.11]{Karlovich03}.
\subsection{The local principle of Simonenko type}
Two functions $a,b\in L^\infty(\Gamma)$ are said to be locally
equivalent at a point $t\in\Gamma$ if
\[
\inf\big\{\|(a-b)c\|_\infty\ :\ c\in C(\Gamma),\ c(t)=1\big\}=0.
\]
\begin{theorem}\label{th:local_principle}
Suppose  $a\in L^\infty(\Gamma)$ and for each $t\in\Gamma$ there exists a function
$a_t\in L^\infty(\Gamma)$ which is locally equivalent to $a$ at $t$. If the
operators $a_tP+Q$ are Fredholm on $L^{p(\cdot)}(\Gamma,w)$ for all $t\in\Gamma$,
then $aP+Q$ is Fredholm on $L^{p(\cdot)}(\Gamma,w)$.
\end{theorem}
For weighted Lebesgue spaces this theorem is known as Simonenko's local
principle \cite{Simonenko65}. It follows from \cite[Theorem~6.13]{Karlovich03}.
\subsection{Wiener-Hopf factorization}
The curve $\Gamma$ divides the complex plane $\mathbb{C}$ into the bounded
simply connected domain $D^+$ and the unbounded domain $D^-$. Without loss
of generality we assume that $0\in D^+$. We say that a function
$a\in L^\infty(\Gamma)$ admits a \textit{Wiener-Hopf factorization on}
$L^{p(\cdot)}(\Gamma,w)$ if $1/a\in L^\infty(\Gamma)$ and $a$ can be written
in the form
\begin{equation}\label{eq:WH}
a(t)=a_-(t)t^\kappa a_+(t)
\quad\mbox{a.e. on}\ \Gamma,
\end{equation}
where $\kappa\in\Z$, and the factors $a_\pm$ enjoy the following properties:
\[
\begin{array}{lll}
({\rm i}) &
a_-\in QL^{p(\cdot)}(\Gamma,w)\stackrel{\cdot}{+}\mathbb{C}, &
1/a_-\in QL^{q(\cdot)}(\Gamma,1/w)\stackrel{\cdot}{+}\mathbb{C},
\\[2mm]
&a_+\in PL^{q(\cdot)}(\Gamma,1/w), &
1/a_+\in PL^{p(\cdot)}(\Gamma,w),
\\[2mm]
({\rm ii}) & \mbox{the operator $(1/a_+)Sa_+I$} &
\mbox{is bounded on $L^{p(\cdot)}(\Gamma,w)$.}
\end{array}
\]
One can prove that the number $\kappa$ is uniquely determined.
\begin{theorem}\label{th:factorization}
A function $a\in L^\infty(\Gamma)$ admits a Wiener-Hopf factorization
\eqref{eq:WH} on $L^{p(\cdot)}(\Gamma,w)$ if and only if the operator $aP+Q$
is Fredholm on $L^{p(\cdot)}(\Gamma,w)$. If $aP+Q$ is Fredholm, then its
index is equal to $-\kappa$.
\end{theorem}
This theorem goes back to Simonenko \cite{Simonenko64,Simonenko68}.
For more about this topic we refer to \cite[Section~6.12]{BK97},
\cite[Section~5.5]{BS06}, \cite[Section~8.3]{GK92} and also to \cite{CG81,LS87}
in the case of weighted Lebesgue spaces. Theorem~\ref{th:factorization} follows
from \cite[Theorem~6.14]{Karlovich03}.
\section{Singular integral operators with $PC$ coefficients}
\label{sect:4}
\subsection{Necessary condition for Fredholmness in case of arbitrary weights}
\label{sec:4.1}
\begin{theorem}[see \cite{Karlovich03}, Theorem~8.1]
\label{th:Nakano-necessity}
Let $\Gamma$ be a Carleson Jordan curve and let $p:\Gamma\to(1,\infty)$
 be a continuous function satisfying \eqref{eq:Dini-Lipschitz}. Suppose
$w:\Gamma\to[0,\infty]$ is an arbitrary weight such that the operator $S$
is bounded on $L^{p(\cdot)}(\Gamma,w)$. If the operator $aP+Q$, where
$a\in PC(\Gamma)$, is Fredholm on $L^{p(\cdot)}(\Gamma,w)$, then
$a(t\pm 0)\ne 0$ and
\[
-
\frac{1}{2\pi}\arg\frac{a(t-0)}{a(t+0)}+\frac{1}{p(t)}+
\theta\alpha_t\left(\frac{1}{2\pi}\log\left|\frac{a(t-0)}{a(t+0)}\right|\right)
+
(1-\theta)
\beta_t\left(\frac{1}{2\pi}\log\left|\frac{a(t-0)}{a(t+0)}\right|\right)
\]
is not an integer number for all $t\in\Gamma$ and all $\theta\in[0,1]$.
\end{theorem}
Notice that from Theorem~\ref{th:Ap} and Lemma~\ref{le:indicator-existence}
it follows that the indicator functions $\alpha_t$ and $\beta_t$ are well
defined for all $t\in\Gamma$.

For standard Lebesgue spaces with Muckenhoupt weights the converse
to Theorem~\ref{th:Nakano-necessity} is also true (see \cite[Proposition~7.3]{BK97}).
\begin{corollary}\label{co:boundedness-necessity}
Let $\Gamma$ be a Carleson Jordan curve and let $p:\Gamma\to(1,\infty)$
be a continuous function satisfying \eqref{eq:Dini-Lipschitz}. If
$w:\Gamma\to[0,\infty]$ is an arbitrary weight such that the operator $S$
is bounded on $L^{p(\cdot)}(\Gamma,w)$, then
\begin{equation}\label{eq:boundedness-necessity-1}
0<1/p(t)+\alpha(V_t^0w),\quad 1/p(t)+\beta(V_t^0w)<1
\end{equation}
for every $t\in\Gamma$.
\end{corollary}
\begin{proof}
From Theorem~\ref{th:Ap} and Lemma~\ref{le:indicator-existence}(b) it follows
that
\begin{equation}\label{eq:boundedness-necessity-2}
0\le 1/p(t)+\alpha(V_t^0w),\quad 1/p(t)+\beta(V_t^0w)\le 1.
\end{equation}
On the other hand, if we take $a=1$, then $aP+Q=I$ is obviously invertible
on $L^{p(\cdot)}(\Gamma,w)$. By Theorem~\ref{th:Nakano-necessity},
\begin{equation}\label{eq:boundedness-necessity-3}
1/p(t)+\theta\alpha(V_t^0w)+(1-\theta)\beta(V_t^0w)\notin\Z
\end{equation}
for all $t\in\Gamma$ and all $\theta\in[0,1]$. Combining
\eqref{eq:boundedness-necessity-2} and \eqref{eq:boundedness-necessity-3}
with $\theta=0$ and $\theta=1$, we arrive at \eqref{eq:boundedness-necessity-1}.
\end{proof}
\begin{corollary}
Let $\Gamma$ be a Carleson Jordan curve and $p:\Gamma\to(1,\infty)$
be a continuous function satisfying \eqref{eq:Dini-Lipschitz}.
Suppose a weight $w$ is given by \eqref{eq:weight}, where $w_1,\dots,w_n\in\W$.
If the Cauchy singular integral operator $S$ is bounded on the space
$L^{p(\cdot)}(\Gamma,w)$, then the condition \eqref{eq:KSS-condition} is
fulfilled.
\end{corollary}
This result follows from Corollary~\ref{co:boundedness-necessity} and
Theorem~\ref{th:radial-MO}. It gives the necessity portion of
Theorem~\ref{th:KSS}.
\subsection{Wiener-Hopf factorization of local representatives}
Fix $t\in\Gamma$.
For a function $a\in PC(\Gamma)$ such that $a^{-1}\in L^\infty(\Gamma)$,
we construct a ``canonical'' function $g_{t,\gamma}$ which is locally equivalent
to $a$ at the point $t\in\Gamma$. The interior and the exterior of the unit circle
can be conformally mapped onto $D^+$ and $D^-$ of $\Gamma$, respectively,
so that the point $1$ is mapped to $t$, and the points $0\in D^+$ and
$\infty\in D^-$ remain fixed. Let $\Lambda_0$ and $\Lambda_\infty$
denote the images of $[0,1]$ and $[1,\infty)\cup\{\infty\}$ under this map.
The curve $\Lambda_0\cup\Lambda_\infty$ joins $0$ to $\infty$ and
meets $\Gamma$ at exactly one point, namely $t$. Let $\arg z$ be a
continuous branch of argument in $\mathbb{C}\setminus(\Lambda_0\cup\Lambda_\infty)$.
For $\gamma\in\mathbb{C}$, define the function $z^\gamma:=|z|^\gamma e^{i\gamma\arg z}$,
where $z\in\mathbb{C}\setminus(\Lambda_0\cup\Lambda_\infty)$. Clearly, $z^\gamma$
is an analytic function in $\mathbb{C}\setminus(\Lambda_0\cup\Lambda_\infty)$. The
restriction of $z^\gamma$ to $\Gamma\setminus\{t\}$ will be denoted by
$g_{t,\gamma}$. Obviously, $g_{t,\gamma}$ is continuous and nonzero on
$\Gamma\setminus\{t\}$. Since $a(t\pm 0)\ne 0$, we can define
$\gamma_t=\gamma\in\mathbb{C}$ by the formulas
\begin{equation}\label{eq:local-representative}
\operatorname{Re}\gamma_t:=\frac{1}{2\pi}\arg\frac{a(t-0)}{a(t+0)},
\quad
\operatorname{Im}\gamma_t:=-\frac{1}{2\pi}\log\left|\frac{a(t-0)}{a(t+0)}\right|,
\end{equation}
where we can take any value of $\arg(a(t-0)/a(t+0))$, which implies that
any two choices of $\operatorname{Re}\gamma_t$ differ by an integer only.
Clearly, there is a constant $c_t\in\mathbb{C}\setminus\{0\}$ such that
$a(t\pm 0)=c_tg_{t,\gamma_t}(t\pm 0)$, which means that $a$ is locally
equivalent to $c_tg_{t,\gamma_t}$ at the point $t\in\Gamma$.

For $t\in\Gamma$ and $\gamma\in\C$, consider the weight
\[
\varphi_{t,\gamma}(\tau):=|(\tau-t)^\gamma|,
\quad
t\in\Gamma\setminus\{t\}.
\]

From \cite[Lemma~7.1]{Karlovich03} we get the following.
\begin{lemma}\label{le:fact-sufficiency}
Let $\Gamma$ be a Carleson Jordan curve and let $p:\Gamma\to(1,\infty)$ be a
continuous function. Suppose $w:\Gamma\to[0,\infty]$ is an arbitrary weight
such that the operator $S$ is bounded on $L^{p(\cdot)}(\Gamma,w)$. If, for
some $k\in\Z$ and $\gamma\in\C$, the operator
$\varphi_{t,k-\gamma}S\varphi_{t,\gamma-k}I$ is bounded on $L^{p(\cdot)}(\Gamma,w)$,
then the function $g_{t,\gamma}$ admits a Wiener-Hopf factorization on
$L^{p(\cdot)}(\Gamma,w)$.
\end{lemma}
\subsection{Fredholm criterion}
Now we are in a position to prove one of the main results of the paper.
\begin{theorem}\label{th:Fredholm}
Let $\Gamma$ be a Carleson Jordan curve satisfying the logarithmic whirl
condition \eqref{eq:spiralic} at each point $t\in\Gamma$ and
$p:\Gamma\to(1,\infty)$ be a continuous function satisfying
\eqref{eq:Dini-Lipschitz}. Suppose a weight $w$ is given by
\eqref{eq:weight}, where for $w_1,\dots,w_n\in\W$ the condition
\eqref{eq:KSS-condition} is fulfilled. The operator $aP+Q$, where $a\in PC(\Gamma)$,
is Fredholm on the weighted variable Lebesgue space $L^{p(\cdot)}(\Gamma,w)$
if and only if $a(t\pm 0)\ne 0$ and
\begin{equation}\label{eq:Fredholm-1}
-\frac{1}{2\pi}\arg\frac{a(t-0)}{a(t+0)}
+\frac{\delta(t)}{2\pi}\log\left|\frac{a(t-0)}{a(t+0)}\right|
+\frac{1}{p(t)}+\theta\mu_t+(1-\theta)\nu_t\notin\Z
\end{equation}
for all $t\in\Gamma$ and all $\theta\in[0,1]$, where
\begin{equation}\label{eq:definition-mu-nu}
\begin{array}{llll}
\mu_{t_k}=m(w_k), & \nu_{t_k}=M(w_k) &\mbox{for}  &k\in\{1,\dots,n\},
\\[2mm]
\mu_t=0, & \nu_t=0 &  \mbox{for}   &t\in\Gamma\setminus\{t_1,\dots,t_n\}.
\end{array}
\end{equation}
\end{theorem}
\begin{proof}
The proof is developed by analogy with the proof of \cite[Theorem~3.3]{Karlovich05}.

\textit{Necessity.}
By Theorem~\ref{th:KSS}, the operator $S$ is bounded on $L^{p(\cdot)}(\Gamma,w)$.
Then from Theorem~\ref{th:Ap} it follows that $w\in A_{p(\cdot)}(\Gamma)$.
Therefore, by Lemma~\ref{le:indicator-existence}(a), the function $V_t^0(\eta_t^xw)$
is regular and submultiplicative for all $t\in\Gamma$ and all $x\in\R$.
Since $\Gamma$ satisfies \eqref{eq:spiralic} for all $t\in\Gamma$, from
Lemma~\ref{le:indicator-calculation} and Theorem~\ref{th:radial-MO} it follows
that the indices of $V_t^0(\eta_t^xw)$ (that is, the  indicator functions
$\alpha_t,\beta_t$ of the pair $(\Gamma,w)$ at $t\in\Gamma$) are calculated by
\begin{equation}\label{eq:Fredholm-N1}
\alpha_{t_k}(x)=m(w_k)+\delta(t)x,
\quad
\beta_{t_k}(x)=M(w_k)+\delta(t)x
\quad(x\in\R)
\end{equation}
for $k\in\{1,\dots,n\}$ and
\begin{equation}\label{eq:Fredholm-N2}
\alpha_t(x)=\delta(t)x,
\quad
\beta_t(x)=\delta(t)x
\quad(x\in\R)
\end{equation}
for $t\in\Gamma\setminus\{t_1,\dots,t_n\}$. Theorem~\ref{th:Nakano-necessity}
and \eqref{eq:Fredholm-N1}--\eqref{eq:Fredholm-N2} imply \eqref{eq:Fredholm-1}
and $a(t\pm 0)$ for all $t\in\Gamma$ and all $\theta\in[0,1]$.

\textit{Sufficiency.} If $aP+Q$ is Fredholm, then, by Theorem~\ref{th:necessity},
$a(t\pm 0)\ne 0$ for all $t\in\Gamma$. Fix an arbitrary $t\in\Gamma$.
Choose $\gamma=\gamma_t\in\C$ as in \eqref{eq:local-representative}.
Then the function $a$ is locally equivalent to $c_tg_{t,\gamma_t}$ at the point
$t\in\Gamma$, where $c_t\in\C\setminus\{0\}$ is some constant.
{From} \eqref{eq:Fredholm-1} it follows that there exists an $s_t\in\Z$ such that
\[
0<s_t-\operatorname{Re}\gamma_t-
\delta(t)\operatorname{Im}\gamma_t+1/p(t)+\theta\mu_t+(1-\theta)\nu_t<1
\]
for all $\theta\in[0,1]$. In particular, if $\theta=1$, then
\begin{equation}\label{eq:Fredholm-2}
0<1/p(t)+\big(s_t-\operatorname{Re}\gamma_t-
\delta(t)\operatorname{Im}\gamma_t\big)+\mu_t;
\end{equation}
if $\theta=0$, then
\begin{equation}\label{eq:Fredholm-3}
1/p(t)+\big(s_t-\operatorname{Re}\gamma_t-
\delta(t)\operatorname{Im}\gamma_t\big)+\nu_t<1.
\end{equation}
Let $\psi_t(x):=
x^{s_t-\operatorname{Re}
\gamma_t-\delta(t)\operatorname{Im}\gamma_t}$ for $x\in(0,|\Gamma|]$
and
\[
\widetilde{w}(\tau):=\psi_t(|\tau-t|)w(\tau),
\quad\tau\in\Gamma.
\]
Clearly, the weight $\widetilde{w}$ is of the form \eqref{eq:weight}.

If $t\in\Gamma\setminus\{t_1,\dots,t_n\}$, then obviously
\begin{equation}\label{eq:Fredholm-4}
m(\psi_t)=M(\psi_t)=s_t-\operatorname{Re}\gamma_t-\delta(t)\operatorname{Im}\gamma_t.
\end{equation}
From \eqref{eq:Fredholm-2}--\eqref{eq:Fredholm-4} we get
\[
0<1/p(t)+m(\psi_t)\le 1/p(t)+M(\psi_t)<1.
\]
Combining these inequalities with Theorem~\ref{th:KSS}, we conclude that
the operator $S$ is bounded on $L^{p(\cdot)}(\Gamma,\widetilde{w})$.

If $t=t_k$ for some $k\in\{1,\dots,n\}$, then one can easily show that
\begin{eqnarray}
m(\psi_{t_k}w_k) &=&
s_{t_k}-\operatorname{Re}
\gamma_{t_k}-\delta(t_k)\operatorname{Im}\gamma_{t_k}+m(w_k),
\label{eq:Fredholm-5}
\\
M(\psi_{t_k}w_k) &=&
s_{t_k}-\operatorname{Re}
\gamma_{t_k}-\delta(t_k)\operatorname{Im}\gamma_{t_k}+M(w_k).
\label{eq:Fredholm-6}
\end{eqnarray}
From \eqref{eq:Fredholm-2}--\eqref{eq:Fredholm-3} and
\eqref{eq:Fredholm-5}--\eqref{eq:Fredholm-6} we obtain
\[
0<1/p(t_k)+m(\psi_{t_k}w_k)\le 1/p(t_k)+M(\psi_{t_k}w_k)<1.
\]
These inequalities and Theorem~\ref{th:KSS} yield the boundedness of the
operator $S$ on $L^{p(\cdot)}(\Gamma,\widetilde{w})$.

In view of the logarithmic whirl condition \eqref{eq:spiralic} we have
\begin{eqnarray*}
\varphi_{t,s_t-\gamma_t}(\tau)
&=&
|\tau-t|^{s_t-\operatorname{Re}\gamma_t}
e^{\operatorname{Im}\gamma_t\arg(\tau-t)}
\\
&=&
|\tau-t|^{s_t-\operatorname{Re}\gamma_t}
e^{-\operatorname{Im}\gamma_t(\delta(t)\log|\tau-t|+O(1))}
\\
&=&
|\tau-t|^{s_t-\operatorname{Re}\gamma_t-\delta(t)\operatorname{Im}\gamma_t}
e^{-\operatorname{Im}\gamma_t O(1)}
\end{eqnarray*}
as $\tau\to t$. Therefore the operator
$\varphi_{t,s_t-\gamma_t}S\varphi_{t,\gamma_t-s_t}I$
is bounded on $L^{p(\cdot)}(\Gamma,w)$. Then, by Lemma~\ref{le:fact-sufficiency},
the function $g_{t,\gamma_t}$ admits a Wiener-Hopf factorization on
$L^{p(\cdot)}(\Gamma,w)$.
Due to Theorem~\ref{th:factorization}, the operator $g_{t,\gamma_t}P+Q$
is Fredholm. Then the operator $c_tg_{t,\gamma_t}P+Q$ is Fredholm, too.
Since the function $c_tg_{t,\gamma_t}$ is locally equivalent to the function
$a$ at every point $t\in\Gamma$, the operator $aP+Q$ is Fredholm on
$L^{p(\cdot)}(\Gamma,w)$ in view of Theorem~\ref{th:local_principle}.
\end{proof}
\begin{remark}
In the case of standard Lebesgue spaces a complete description of the set
of all $\gamma\in\C$ such that the operator $S$ is bounded on
$L^p(\Gamma,\varphi_{t,\gamma}w)$ is known in terms of the indicator
functions $\alpha_t$ and $\beta_t$ (see \cite[Sections~3.6--3.7]{BK97}).
This description allowed the authors of \cite{BK97} to consider the case
of Carleson curves that may not satisfy \eqref{eq:spiralic} and may have
distinct spirality indices $\alpha(W_t\eta_t)$ and $\beta(W_t\eta_t)$.
However, the only result on the boundedness of $S$ on weighted variable
Lebesgue spaces in our disposal is Theorem~\ref{th:KSS}, which is not
applicable to weights of the form $\varphi_{t,\gamma}w$ unless a Carleson
curve is sufficiently nice. If a Carleson curve satisfies \eqref{eq:spiralic},
then $\varphi_{t,\gamma}$ is equivalent to $\psi_{t,\lambda}(\tau):=|\tau-t|^\lambda$
for some $\lambda\in\R$ and one can apply Theorem~\ref{th:KSS} to the weight
$\psi_{t,\lambda}w$. Therefore, to treat the case of arbitrary Carleson curves
by this method, a more general boundedness result than Theorem~\ref{th:KSS}
is needed.
\end{remark}
\subsection{Criterion for the closedness of the image}
\begin{theorem}\label{th:criterion-closedness}
Let $\Gamma$ be a Carleson Jordan curve satisfying the logarithmic whirl
condition \eqref{eq:spiralic} at each point $t\in\Gamma$ and
$p:\Gamma\to(1,\infty)$ be a continuous function satisfying
\eqref{eq:Dini-Lipschitz}. Suppose a weight $w$ is given by
\eqref{eq:weight}, where for $w_1,\dots,w_n\in\W$ the condition
\eqref{eq:KSS-condition} is fulfilled. Suppose $a\in PC(\Gamma)$ has
finitely many jumps and $a(t\pm 0)\ne 0$ for all $t\in\Gamma$. Then
the image of $aP+Q$ is closed in $L^{p(\cdot)}(\Gamma,w)$ if and only
if \eqref{eq:Fredholm-1} holds for all $t\in\Gamma$.
\end{theorem}
\begin{proof}
The idea of the proof is borrowed from \cite[Proposition~7.16]{BK97}
(see also \cite[Theorem~4.3]{Karlovich07}).
The sufficiency part follows from Theorem~\ref{th:Fredholm}.

Let us prove the necessity part. Assume that $a(t\pm 0)\ne 0$ for all
$t\in\Gamma$. Since the number of jumps, that is, the points $t\in\Gamma$ at
which $a(t-0)\ne a(t+0)$, is finite, it is clear that
\[
\begin{split}
&
-\frac{1}{2\pi}\arg\frac{a(t-0)}{a(t+0)}
+\frac{\delta(t)}{2\pi}\log\left|\frac{a(t-0)}{a(t+0)}\right|
+\frac{1}{1+\eps}\notin\Z,
\\[3mm]
&
-\frac{1}{2\pi}\arg\frac{a(t+0)}{a(t-0)}
+\frac{\delta(t)}{2\pi}\log\left|\frac{a(t+0)}{a(t-0)}\right|
+\frac{1}{1+\eps}\notin\Z
\end{split}
\]
for all $t\in\Gamma$ and all sufficiently small $\eps>0$. By
Theorem~\ref{th:Fredholm}, the operators $aP+Q$
and $a^{-1}P+Q$ are Fredholm on the Lebesgue space $L^{1+\eps}(\Gamma)$
whenever $\eps>0$ is sufficiently small.

If (\ref{eq:KSS-condition}) holds, then there exists a number $\eps>0$ such that
\[
0<(1/p(t_k)+m(w_k))(1+\eps)\le (1/p(t_k)+M(w_k))(1+\eps)<1
\]
for all $k\in\{1,\dots,n\}$. It is easy to see that
\[
m(w_k^{1+\eps})=(1+\eps)m(w_k),
\quad
M(w_k^{1+\eps})=(1+\eps)M(w_k).
\]
Hence, by Theorem~\ref{th:KSS}, the operator
$S$ is bounded on $L^{p(\cdot)/(1+\eps)}(\Gamma,w^{1+\eps})$.
Taking into account this observation, one can prove as in
\cite[Lemma~4.1]{Karlovich07} that there is an $\eps_0>0$ such that
\[
L^{p(\cdot)}(\Gamma,w)\subset L^{1+\eps_0}(\Gamma),
\quad
L^{q(\cdot)}(\Gamma,1/w)\subset L^{1+\eps_0}(\Gamma)
\]
and $aP+Q$, $a^{-1}P+Q$ are Fredholm on $L^{1+\eps_0}(\Gamma)$. Then
\begin{equation}\label{eq:criterion-closedness-1}
n\big(aP+Q;L^{p(\cdot)}(\Gamma,w)\big)
\le
n\big(aP+Q;L^{1+\eps_0}(\Gamma)\big)<\infty,
\end{equation}
and, by duality (see \cite[Lemma~3.8]{Karlovich07}),
\begin{equation}\label{eq:criterion-closedness-2}
\begin{split}
d\big(aP+Q;L^{p(\cdot)}(\Gamma,w)\big)
&=
n\big(a^{-1}P+Q;L^{q(\cdot)}(\Gamma,1/w)\big)
\\
&\le
n\big(a^{-1}P+Q;L^{1+\eps_0}(\Gamma)\big)<\infty.
\end{split}
\end{equation}
If (\ref{eq:Fredholm-1}) does not hold, then $aP+Q$ is not Fredholm on
$L^{p(\cdot)}(\Gamma,w)$ in view of Theorem~\ref{th:Fredholm}.
From this fact and (\ref{eq:criterion-closedness-1})--(\ref{eq:criterion-closedness-2})
we conclude that the image of $aP+Q$ is not closed in $L^{p(\cdot)}(\Gamma,w)$,
which contradicts the hypothesis.
\end{proof}
\section{Tools for the construction of the symbol calculus}
\label{sect:5}
\subsection{The Allan-Douglas local principle}
Let $B$ be a Banach algebra with identity. A subalgebra $Z$ of $B$ is said to
be a \textit{central subalgebra} if $zb=bz$ for all $z\in Z$ and all $b\in B$.
\begin{theorem}[see \cite{BS06}, Theorem~1.35(a)]
\label{th:AllanDouglas}
Let $B$ be a Banach algebra with identity $e$ and let $Z$ be a closed central
subalgebra of $B$ containing $e$. Let $M(Z)$ be the maximal ideal space of $Z$,
and for $\omega\in M(Z)$, let $J_\omega$ refer to the smallest closed two-sided
ideal of $B$ containing the ideal $\omega$. Then an element $b$ is invertible
in $B$ if and only if $b+J_\omega$ is invertible in the quotient algebra
$B/J_\omega$ for all $\omega\in M(Z)$.
\end{theorem}
The algebra $B/J_\omega$ is referred to as the \textit{local algebra} of $B$
at $\omega\in M(Z)$ and the spectrum of $b+J_\omega$
in $B/J_\omega$ is called the \textit{local spectrum} of $b$ at $\omega\in M(Z)$.
\subsection{The two projections theorem}
Recall that an element $p$ of a Banach algebra is called an \textit{idempotent}
(or, somewhat loosely, also a \textit{projection}), if $p^2=p$.

The following two projections theorem was obtained by Finck,
Roch, Silbermann \cite{FRS93} and Gohberg, Krupnik \cite{GK93}
(see also \cite[Section~8.3]{BK97}).
\begin{theorem}\label{th:2proj}
Let $B$ be a Banach algebra with identity $e$, let $\cC$ be a Banach subalgebra
of $B$ which contains $e$ and is isomorphic to ${\mathbb{C}}^{N \times N}$,
and let $p$ and $q$ be two idempotent elements in $B$ such that
$cp=pc$ and $cq=qc$ for all $c \in \cC$. Let $A=\alg(\cC,p,q)$
be the smallest closed subalgebra of $B$ containing $\cC,p,q$. Put
\[
x=pqp+(e-p)(e-q)(e-p),
\]
denote by $\mathrm{sp}\,x$ the spectrum of $x$ in $B$, and suppose the points
$0$ and $1$ are not isolated points of $\mathrm{sp}\,x$. Then
\begin{enumerate}
\item[{\rm (a)}]
for each $z \in \mathrm{sp}\,x$ the map $\sigma_{z}$ of $\cC \cup \{p,q\}$
into the algebra ${\mathbb{C}}^{2N\times 2N}$ of all complex $2N\times 2N$
matrices defined by
\begin{eqnarray*}
&&
\sigma_{z}c=\left[
\begin{array}{cc}
c & O\\
O & c
\end{array}
\right],
\sigma_{z}p=\left[
\begin{array}{cc}
E & O\\
O & O
\end{array}
\right],
\sigma_{z}q=\left[
\begin{array}{cc}
z E & \sqrt{z(1-z)}E \\
\sqrt{z(1-z)}E  & (1-z)E
\end{array}
\right],
\end{eqnarray*}
where $c\in \cC$, $E$ and $O$ denote the $N \times N$ identity and zero matrices,
respectively, and $\sqrt{z(1-z)}$ denotes any complex number whose square is $z(1-z)$,
extends to a Banach algebra homomorphism
\[
\sigma_{z}: A \to {\mathbb{C}}^{2N \times 2N};
\]

\item[{\rm (b)}]
every element $a$ of the algebra $A$ is invertible in the algebra $B$ if
and only if
\[
\det \sigma_{z} a \neq 0 \quad\mbox{for all}\quad z \in \mathrm{sp}\,x;
\]

\item[{\rm (c)}]
the algebra $A$ is inverse closed in $B$ if and only if the spectrum of
$x$ in $A$ coincides with the spectrum of $x$ in $B$.
\end{enumerate}
\end{theorem}
A further generalization of the above result to the case of $n\ge 3$ projections
is contained in \cite[Section~8.4]{BK97}.
\section{Algebra of singular integral operators}
\label{sect:6}
\subsection{Operators of local type}
In this section we will suppose that $\Gamma$ is a Carleson curve satisfying
the logarithmic whirl condition \eqref{eq:spiralic} at each point $t\in\Gamma$,
$p:\Gamma\to(1,\infty)$ is a continuous function satisfying \eqref{eq:Dini-Lipschitz},
$w$ is a weight of the form \eqref{eq:weight} with $w_1,\dots,w_n\in\W$ satisfying
\eqref{eq:KSS-condition}. Under these conditions the operator $S$ (defined
elementwise) is bounded on $L_N^{p(\cdot)}(\Gamma,w)$, where $N\ge 1$ (see
Theorem~\ref{th:KSS}).

Let $\cB:=\cB(L_N^{p(\cdot)}(\Gamma,w))$  be the Banach algebra of all bounded
linear operators on $L_N^{p(\cdot)}(\Gamma,w)$ and let $\cK:=\cK(L_N^{p(\cdot)}(\Gamma,w))$
be the closed ideal of all compact operators on $L_N^{p(\cdot)}(\Gamma,w)$.
We will denote by $\cB^\pi$ the Calkin algebra $\cB/\cK$ and by $A^\pi$
the coset $A+\cK$ for any operator $A\in\cB$. An operator $A\in\cB$ is said
to be of \textit{local type} if $A\diag\{c,\dots,c\}I-\diag\{c,\dots,c\}A\in\cK$
for every continuous function $c$ on $\Gamma$. This notion goes back to
Simonenko \cite{Simonenko65,SM86}. It is easy to see that the set $\cL$ of all
operators of local type is a Banach subalgebra of $\cB$.
\begin{lemma}\label{le:OLT-embedding}
We have
\begin{equation}\label{eq:OLT-embedding-1}
\cK\subset\alg(S,PC,L_N^{p(\cdot)}(\Gamma,w))\subset\cL.
\end{equation}
\end{lemma}
\begin{proof}
Let $\alg(S,C,L_N^{p(\cdot)}(\Gamma,w))$ be the smallest closed subalgebra
of $\cB$ containing the operators of multiplication by continuous matrix
functions and the operator $S$. Obviously,
\begin{equation}\label{eq:OLT-embedding-2}
\alg(S,C,L_N^{p(\cdot)}(\Gamma,w))\subset\alg(S,PC,L_N^{p(\cdot)}(\Gamma,w)).
\end{equation}
On the other hand, by analogy with \cite[Lemma~9.1]{Karlovich96} or
\cite[Lemma~5.1]{Karlovich06} one can show that
\begin{equation}\label{eq:OLT-embedding-3}
\cK\subset\alg(S,C,L_N^{p(\cdot)}(\Gamma,w)).
\end{equation}
Combining \eqref{eq:OLT-embedding-2} and \eqref{eq:OLT-embedding-3},
we arrive at the first embedding in \eqref{eq:OLT-embedding-1}.

If $a\in PC_{N\times N}(\Gamma)$ and $c\in C(\Gamma)$, then the operators
$aI$ and $\diag\{c,\dots,c\}I$ obviously commute. Then $aI$ is of
local type. By \cite[Lemma~6.5]{Karlovich03}, the operator
$\diag\{c,\dots,c\}S-S\diag\{c,\dots,c\}I$ is compact. Thus, $S$
is of local type, too. Since the generators of
$\alg(S,PC,L_N^{p(\cdot)}(\Gamma,w))$ are of local type, each element
of this algebra is of local type, which proves the second embedding in
\eqref{eq:OLT-embedding-1}.
\end{proof}
\begin{lemma}\label{le:OLT-Fredholmness-invertibility}
An operator $A\in\cL$ is Fredholm if and only if the coset $A^\pi$ is invertible
in the quotient algebra $\cL^\pi:=\cL/\cK$.
\end{lemma}
The proof is straightforward.
\subsection{Localization}
{From} Lemma~\ref{le:OLT-embedding} we deduce that the quotient algebras
\[
\alg^\pi(S,PC,L_N^{p(\cdot)}(\Gamma,w)):=\alg(S,PC,L_N^{p(\cdot)}(\Gamma,w))/\cK
\]
and $\cL^\pi:=\cL/\cK$ are well defined. We will study the invertibility of
an element $A^\pi$ of $\alg^\pi(S,PC,L_N^{p(\cdot)}(\Gamma,w))$ in the larger
algebra $\cL^\pi$ by using Theorem~\ref{th:AllanDouglas}. To this end, consider
\[
\cZ^\pi:=\big\{(\diag\{c,\dots,c\}I)^\pi:c\in C(\Gamma)\big\}.
\]
{From} the definition of $\cL$ it follows that $\cZ^\pi$ is a central subalgebra
of $\cL^\pi$. The maximal ideal space $M(\cZ^\pi)$ of $\cZ^\pi$ may be
identified with the curve $\Gamma$ via the Gelfand map $\cG$ given by
\[
\cG:\cZ^\pi\to C(\Gamma),
\quad
\big(\cG(\diag\{c,\dots,c\}I)^\pi\big)(t)=c(t)
\quad (t\in\Gamma).
\]
In accordance with Theorem~\ref{th:AllanDouglas}, for every $t\in\Gamma$
we define $\cJ_t\subset\cL^\pi$ as the smallest closed two-sided ideal of
$\cL^\pi$ containing the set
\[
\big\{(\diag\{c,\dots,c\}I)^\pi\ :\ c\in C(\Gamma),\ c(t)=0\big\}.
\]

Consider a function $\chi_t\in PC(\Gamma)$ which is continuous
on $\Gamma\setminus\{t\}$ and satisfies $\chi_t(t-0)=0$ and $\chi_t(t+0)=1$.
For $a\in PC_{N\times N}(\Gamma)$ define $a_t\in PC_{N\times N}(\Gamma)$ by
\begin{equation}\label{eq:at}
a_t:=a(t-0)(1-\chi_t)+a(t+0)\chi_t.
\end{equation}
Clearly $(aI)^\pi-(a_tI)^\pi\in\cJ_t$. Hence, for any operator
$A\in\alg(S,PC,L_N^{p(\cdot)}(\Gamma,w))$, the coset $A^\pi+\cJ_t$ belongs
to the smallest closed subalgebra $\cA_t$ of $\cL^\pi/\cJ_t$ containing the
cosets
\begin{equation}\label{eq:definition-pq}
p:=P^\pi+\cJ_t,
\quad
q:=(\diag\{\chi_t,\dots,\chi_t\}I)^\pi+\cJ_t
\end{equation}
and the algebra
\begin{equation}\label{eq:definition-C}
\cC:=\big\{(cI)^\pi+\cJ_t\ : \ c\in\C^{N\times N}\big\}.
\end{equation}
The latter algebra is obviously isomorphic to $\C^{N\times N}$,
so $\cC$ and $\C^{N\times N}$ can (and will) be identified with each other.
It is easy to see that
\begin{equation}\label{eq:projections}
p^2=p, \quad
q^2=q, \quad
pc=cp, \quad
qc=cq
\end{equation}
for all $c\in\cC$. To apply Theorem~\ref{th:2proj} to the algebras
$\cL^\pi/\cJ_t$ and $\cA_t=\alg(\cC,p,q)$, we need to identify the spectrum
of the element
\begin{equation}\label{eq:element-x}
\begin{split}
x &:= pqp+(e-p)(e-q)(e-p)
\\
&= \big(P\diag\{\chi_t,\dots,\chi_t\}P+Q\diag\{1-\chi_t,\dots,1-\chi_t\}Q\big)^\pi+\cJ_t
\end{split}
\end{equation}
in the algebra $\cL^\pi/\cJ_t$.
\subsection{Spiralic horns}
Given two real numbers $a,b$ satisfying $0<a\le b<1$, two complex numbers $z_1,z_2$,
and a real number $\delta$, we define the \textit{spiralic horn} between
$z_1$ and $z_2$ as the set
\[
\begin{split}
&\cS(z_1,z_2;\delta;a,b):=\{z_1,z_2\}\cup
\\
&\cup\left\{u\in\C\setminus\{z_1,z_2\}:
\frac{1}{2\pi}\left(
\arg\frac{u-z_1}{u-z_2}-\delta\log\left|\frac{u-z_1}{u-z_2}\right|
\right)
\in[a,b]+\Z\right\}.
\end{split}
\]
If $z_1=z_2$, then $\cS(z_1,z_2;\delta;a,b)$ degenerates to the point $z_1$.
Assume $z_1 \ne z_2$.
Then the set $\cS(z_1,z_2;0;1/2,1/2)$ is the segment between $z_1$
and $z_2$. If $a\ne 1/2$, then $\cS(z_1,z_2;0;a,a)$ is the circular arc
between $z_1$ and $z_2$. If $a<1/2$ (resp. $a>1/2$), then one sees the straight
line between $z_1$ and $z_2$ under the angle $2\pi a$ (resp. $2\pi(1-a)$)
and running through the arc from $z_1$ to $z_2$ this straight line is located
at the left-hand side (resp. right-hand side).
The importance of these arcs in the Fredholm theory of singular integral
operators was first observed by Widom \cite{Widom60}, they were exploited
intensively by Gohberg and Krupnik \cite{GK71}, \cite[Chap.~9]{GK92}.

If $z_1\ne z_2$ and $a<b$, then $\cS(z_1,z_2;0;a,b)$ is an ordinary horn
bounded by the circular arcs $\cS(z_1,z_2;0;a,a)$ and $\cS(z_1,z_2;0;b,b)$.
If $\delta\ne 0$, then $\cS(z_1,z_2;\delta;a,a)$ is a logarithmic double spiral.
If $a<b$, then $\cS(z_1,z_2;\delta;a,b)$ is the union of logarithmic double
spirals
\[
\cS(z_1,z_2;\delta;a,b)=\bigcup_{\lambda\in[a,b]}\cS(z_1,z_2;\delta;\lambda,\lambda).
\]
Hence this set is the closed set between two logarithmic double spirals.
It was introduced in \cite{BK95} (see also \cite[Sections~7.3--7.6]{BK97} and
\cite{BKR96,BK01}).
\subsection{The local spectrum}
Now we are ready to identify the spectrum of $x$ in the algebra
$B=\cL^\pi/\cJ_t$.
\begin{lemma}\label{le:local-spectrum}
Let $\chi_t\in PC(\Gamma)$ be a continuous function on $\Gamma\setminus\{t\}$
such that
\[
\chi_t(t-0)=0,
\
\chi_t(t+0)=1,
\
\chi_t(\Gamma\setminus\{t\})\cap\cS(0,1;\delta(t);1/p(t)+\mu_t,1/p(t)+\nu_t)=\emptyset,
\]
where $\mu_t$ and $\nu_t$ are defined by \eqref{eq:definition-mu-nu}.
Then the spectrum of the element $x$ given by \eqref{eq:element-x}
in the algebra $B=\cL^\pi/\cJ_t$ coincides with
$\cS(0,1;\delta(t);1/p(t)+\mu_t,1/p(t)+\nu_t)$.
\end{lemma}
\begin{proof}
From Theorem~\ref{th:Fredholm} we immediately get that the set of all
$\lambda\in\C$ such that the operator $(\chi_t-\lambda)P+Q$
is not Fredholm on $L^{p(\cdot)}(\Gamma,w)$ coincides with
$\cS(0,1;\delta(t);1/p(t)+\mu_t,1/p(t)+\nu_t)$.
After this observation the proof can be developed by a literal repetition
of the proof of \cite[Lemma~9.4]{Karlovich96} with the Boyd indices $\alpha_M$
and $\beta_M$ of an Orlicz space considered there replaced by
the numbers $1/p(t)+\mu_t$ and $1/p(t)+\nu_t$, respectively.
\end{proof}
\subsection{Construction of the symbol calculus}
Now we are in a position to prove the main result of the paper.
\begin{theorem}\label{th:main}
Let $N\in\mathbb{N}$, let $\Gamma$ be a Carleson Jordan curve satisfying the
logarithmic whirl condition \eqref{eq:spiralic} at every point $t\in\Gamma$,
let $p:\Gamma\to(1,\infty)$ be a continuous function satisfying
\eqref{eq:Dini-Lipschitz}, and let $w:\Gamma\to[0,\infty]$ be a
radial oscillating weight given by \eqref{eq:weight} with $w_1,\dots,w_n\in\W$
satisfying \eqref{eq:KSS-condition}. Put
\[
\begin{array}{llll}
\mu_{t_k}=m(w_k), & \nu_{t_k}=M(w_k) &\mbox{for}  &k\in\{1,\dots,n\},
\\[2mm]
\mu_t=0, & \nu_t=0 &  \mbox{for}   &t\in\Gamma\setminus\{t_1,\dots,t_n\}
\end{array}
\]
and define the ``spiralic horn bundle" by
\[
\cM:=\bigcup_{t\in\Gamma}\big(\{t\}\times
\cS(0,1;\delta(t);1/p(t)+\mu_t,1/p(t)+\nu_t)\big).
\]
\begin{enumerate}
\item[{\rm(a)}]
For each point $(t,z)\in\cM$, the map
\[
\sigma_{t,z}:\{S\}\cup\{aI\ :\ a\in PC_{N\times N}(\Gamma)\}\to\C^{2N\times 2N}
\]
given by
\[
\sigma_{t,z}(S)
=
\left[\begin{array}{cc}
E & O \\ O & -E
\end{array}\right],
\quad
\sigma_{t,z}(aI)
=
\]
\[
=\left[\begin{array}{cc}
a(t+0)z+a(t-0)(1-z) & (a(t+0)-a(t-0))\sqrt{z(1-z)}
\\[2mm]
(a(t+0)-a(t-0))\sqrt{z(1-z)} & a(t+0)(1-z)+a(t-0)z
\end{array}\right],
\]
where $E$ and $O$ denote the $N \times N$ identity and zero matrices,
respectively, and $\sqrt{z(1-z)}$ denotes any complex number whose square
is $z(1-z)$, extends to a Banach algebra homomorphism
\[
\sigma_{t,z}:\alg(S,PC,L_N^{p(\cdot)}(\Gamma,w))\to\C^{2N\times 2N}
\]
with the property that $\sigma_{t,z}(K)$ is the $2N\times 2N$ zero matrix
whenever $K$ is a compact operator on $L_N^{p(\cdot)}(\Gamma,w)$.

\item[{\rm(b)}]
An operator $A\in\alg(S,PC,L_N^{p(\cdot)}(\Gamma,w))$ is Fredholm
on $L_N^{p(\cdot)}(\Gamma,w)$ if and only if
\[
\det\sigma_{t,z}(A)\ne 0
\quad\mbox{for all}\quad (t,z)\in\cM.
\]

\item[{\rm(c)}]
The quotient algebra $\alg^\pi(S,PC,L_N^{p(\cdot)}(\Gamma,w))$ is inverse closed
in the Calkin algebra $\cB^\pi$, that is, if a coset
$A^\pi\in\alg^\pi(S,PC,L_N^{p(\cdot)}(\Gamma,w))$ is invertible in $\cB^\pi$, then
$(A^\pi)^{-1}\in\alg^\pi(S,PC,L_N^{p(\cdot)}(\Gamma,w))$.
\end{enumerate}
\end{theorem}
\begin{proof}
The idea of the proof is borrowed from \cite[Section~8.5]{BK97}. Here we
follow the proof of \cite[Theorem~5.1]{Karlovich05} and \cite[Theorem~5.4]{Karlovich06}.
Since the latter two sources may not be readily available, we give a
selfcontained proof.

Fix $t\in\Gamma$ and choose a function $\chi_t\in PC(\Gamma)$ as in the
assumptions of Lemma~\ref{le:local-spectrum}. From \eqref{eq:projections}
and Lemma~\ref{le:local-spectrum} we deduce that the algebras $\cL^\pi/\cJ_t$
and $\cA_t=\alg(\cC,p,q)$, where $p,q$, and $\cC$ are given by
\eqref{eq:definition-pq} and \eqref{eq:definition-C}, respectively,
satisfy all the assumptions of Theorem~\ref{th:2proj}.

(a) By Theorem~\ref{th:2proj}(a), for every
$z\in\cS(0,1;\delta(t);1/p(t)+\mu_t,1/p(t)+\nu_t)$, the map
\[
\sigma_{t,z}=\sigma_z\circ\pi_t:\alg(S,PC,L_N^{p(\cdot)}(\Gamma,w))\to\C^{2N\times 2N},
\]
where $\sigma_z$ is given in Theorem~\ref{th:2proj}(a) and $\pi_t$ acts by
the rule $A\mapsto A^\pi+\cJ_t$ for every $A\in\alg(S,PC,L_N^{p(\cdot)}(\Gamma,w))$,
is a well defined Banach algebra homomorphism. It is easy to check that
\[
\sigma_{t,z}(S)=2\sigma_zp-\sigma_ze
=
\left[\begin{array}{cc}
E & O\\
O & -E
\end{array}\right].
\]
If $a\in PC_{N\times N}(\Gamma)$, then from \eqref{eq:at} and
$(aI-a_tI)^\pi\in\cJ_t$ it follows that
\[
\begin{split}
\sigma_{t,z}(aI)&=\sigma_{t,z}(a_tI)=
\sigma_z(a(t-0))\sigma_z(e-q)+\sigma_z(a(t+0))\sigma_zq
\\
&=
\left[\begin{array}{cc}
a(t+0)z+a(t-0)(1-z) & (a(t+0)-a(t-0))\sqrt{z(1-z)}
\\[2mm]
(a(t+0)-a(t-0))\sqrt{z(1-z)} & a(t+0)(1-z)+a(t-0)z
\end{array}\right].
\end{split}
\]
From Lemma~\ref{le:OLT-embedding} it follows that $\pi_t(K)$ is correctly defined
for every $K\in\cK$ and $\pi_t(K)=\cJ_t$. Hence
\[
\sigma_{t,z}(K)=\sigma_z(0)=\left[\begin{array}{cc}
O & O\\O & O
\end{array}\right].
\]
Part (a) is proved.

\medskip
(b) From Lemma~\ref{le:OLT-Fredholmness-invertibility} it follows that the
Fredholmness of an operator $A$ from $\alg(S,PC,L_N^{p(\cdot)}(\Gamma,w))$
is equivalent to the invertibility of $A^\pi\in\cL^\pi$. By
Theorem~\ref{th:AllanDouglas}, the latter is equivalent to the invertibility
of $\pi_t(A)=A^\pi+\cJ_t$ in $\cL^\pi/\cJ_t$ for every $t\in\Gamma$. By
Theorem~\ref{th:2proj}(b), this is equivalent to
\begin{equation}\label{eq:main-1}
\det\sigma_{t,z}(A)=\det\sigma_z\pi_t(A)\ne 0
\quad\mbox{for all}\quad
(t,z)\in\cM.
\end{equation}
Part (b) is proved.

\medskip
(c) Since the compact simply connected set
$\cS(0,1;\delta(t);1/p(t)+\mu_t,1/p(t)+\nu_t)$
does not separate the complex plane, it follows that the spectra of the
element $x$ given by \eqref{eq:element-x} in the algebras $\cL^\pi/\cJ_t$
and $\cA_t$ coincide. Suppose $A$ belongs to $\alg(S,PC,L_N^{p(\cdot)}(\Gamma,w))$.
If $A^\pi$ is invertible in $\cB^\pi$, then
\eqref{eq:main-1} is fulfilled. Consequently, by Theorem~\ref{th:2proj}(b), (c),
$\pi_t(A)=A^\pi+\cJ_t$ is invertible in $\cA_t$ for every $t\in\Gamma$.
Applying the Allan-Douglas local principle (Theorem~\ref{th:AllanDouglas})
to the algebra $\alg^\pi(S,PC,L_N^{p(\cdot)}(\Gamma,w))$, its central subalgebra
$\cZ^\pi$ and the ideals $\cJ_t$, we obtain that $A^\pi$ is invertible
in $\alg^\pi(S,PC,L_N^{p(\cdot)}(\Gamma,w))$. Thus,
$\alg^\pi(S,PC,L_N^{p(\cdot)}(\Gamma,w))$ is inverse closed in $\cB^\pi$.
\end{proof}
For constant $p$, Khvedelidze weights, and piecewise Lyapunov curves,
Theorem~\ref{th:main} was obtained in \cite{GK71} by a different
method. For constant $p$, arbitrary Muckenhoupt weights,
and arbitrary Carleson curves, an analogue of this result is contained in
\cite[Section~8.5]{BK97}. In the case of variable exponent $p:\Gamma\to(1,\infty)$
and Khvedelidze weights, an earlier version of this theorem for Lyapunov curves
or Radon curves without cusps is in \cite[Theorem~5.4]{Karlovich06} and
for Carleson curves satisfying \eqref{eq:spiralic} is in
\cite[Theorem~5.1]{Karlovich05}.
\subsection{Final remarks}
\begin{remark}[On index formulas]
We do not consider formulas for the index of an operator in the algebra
$\alg(S,PC,L_N^{p(\cdot)}(\Gamma,w))$
since the approach to the study of Banach algebras of SIOs
based on the Allan-Douglas local principle and the two projections
theorem does not allow us to get formulas for the index of an arbitrary
operator in the Banach algebra of SIOs with piecewise continuous coefficients.
These formulas can be obtained similarly to the classical situation
considered by Gohberg and Krupnik \cite{GK71} (see also \cite[Chap.~10]{BK97}).
For reflexive Orlicz spaces over Carleson curves with logarithmic whirl points
this was done by the author \cite{K98-index}. In the case of variable
Lebesgue spaces with radial oscillating weights weights over Carleson curves with
logarithmic whirl points, the index formulas are almost the same as in \cite{K98-index}.
It is only necessary to replace the Boyd indices $\alpha_M$ and $\beta_M$ of an
Orlicz space $L^M$ by the numbers $1/p(t_k)+m(w_k)$ and $1/p(t_k)+M(w_k)$
(at the nodes $t_k$ of the weight), respectively,  or both Boyd indices
by $1/p(t)$ (at all other points $t$) in corresponding index formulas.
\end{remark}
\begin{remark}[On semi-Fredholm operators]
Since one has Theorem~\ref{th:criterion-closedness} at hands, by using ideas
of Spitkovsky \cite{Spitkovsky92}, one can prove that if $a,b\in PC_{N\times N}(\Gamma)$,
then $aP+bQ$ is semi-Fredholm if and only if it is Fredholm. With the help of
the linear dilation procedure (see e.g. \cite{GK71}) this fact can be extended to
the operators of the form $\sum_i\prod_j (a_{ij}P+b_{ij}Q)$ with
$a_{ij},b_{ij}\in PC_{N\times N}(\Gamma)$. Since the property
of an operator to be Fredholm or semi-Fredholm is stable under small perturbations,
we finally arrive at the following result. Its proof, given
for Khvedelidze weights in \cite{Karlovich07}, works also for weights
considered in this paper.
\end{remark}
\begin{theorem}
Suppose that all the hypotheses of Theorem~\ref{th:main} are fulfilled.
If an operator in the
algebra $\alg(S,PC;L_N^{p(\cdot)}(\Gamma,w))$ is semi-Fredholm, then
it is Fredholm.
\end{theorem}
\begin{remark}[On coefficients beyond $PC(\Gamma)$]
The approach of this paper is based on the Wiener-Hopf factorization, local
principles, and the two-projection theorem. Note that there is also another
approach to study SIOs on standard Lebesgue spaces $L^p(\Gamma,w)$ with slowly
oscillating (i.e. not arbitrary!) Muckenhoupt weights over slowly oscillating (i.e. not all!)
Carleson curves. It is based based on the technique of pseudodifferential
operators and limit operators developed by Rabinovich and his coauthors
(see e.g. \cite{Rabinovich96} and \cite{BKR96,BKR00}). This approach allows one
to study SIOs not only with piecewise continuous coefficients, but also with
slowly oscillating coefficients. Notice that a satisfactory theory of
pseudodifferential operators in the setting of variable Lebesgue spaces
did not exist a couple of years ago. Very recently Rabinovich and S. Samko \cite{RS08}
have started to develop such a theory, however they do not consider in that paper
SIOs on curves.
\end{remark}
\begin{remark}[On more general variable exponents]
It is well known that the boundedness of $S$ is guaranteed as soon as a
suitable maximal function is bounded on $L^{p(\cdot)}(\Gamma,w)$.
Lerner \cite{Lerner05}, among other things, observed that if
$p(x)=\alpha+\sin(\log\log(1/|x|)\chi_E(x))$, where $\alpha>2$ is some
constant and $\chi_E$ is the characteristic function of the ball
$E:=\{x\in\R^n:|x|\le 1/e\}$, then the Hardy-Littlewood maximal function
is bounded on $L^{p(\cdot)}(\R^n)$. Clearly, the exponent $p$ in this
example is discontinuous at the origin, so it does not satisfy (an $\R^n$
analog of) the condition \eqref{eq:Dini-Lipschitz}. This exponent belongs
to the class of pointwise multipliers for $BMO$ (the space of functions
of bounded mean oscillation). Kapanadze and Kopaliani \cite{KK08} proved
that if $p$ belongs to $VMO^{1/|\log|}$, a weighted space of functions
with vanishing mean oscillation, then the Hardy-Littlewood maximal function
is bounded on variable Lebesgue space $L^{p(\cdot)}(\Omega)$ over a bounded
domain $\Omega$. Thus, I believe that necessary and sufficient
conditions for the boundedness of the Cauchy singular integral operator
(and other singular integrals and maximal functions) on weighted variable
Lebesgue spaces $L^{p(\cdot)}(\Gamma,w)$ should be formulated in terms of
integral means of the exponent $p$ (i.e., in $BMO$ terms), but not in
pointwise terms like \eqref{eq:Dini-Lipschitz}. It is natural that the number
$1/p(t)$ in \eqref{eq:Fredholm-1} under eventual more general boundedness
hypotheses for $S$ should be replaced by a pair of indices of a suitable
submultiplicative function (see e.g. \cite[Section~4.4]{Karlovich03}).
I believe that, in general, these indices may be different and the following is true.
\end{remark}
\begin{conjecture}\label{conjecture}
Suppose $\mathbb{T}$ is the unit circle and $a:\mathbb{T}\to\C$
be a continuous on $\mathbb{T}\setminus\{1\}$ function such that
$a(1-0)=0$ and $a(1+0)=1$. There is a measurable function $p:\Gamma\to(1,\infty)$
such that the operator $S$ is bounded on $L^{p(\cdot)}(\mathbb{T})$ and the essential
spectrum
\[
\sigma_{{\rm ess}}(aP+Q):=
\{\lambda\in\C: (aP+Q)-\lambda I \mbox{ is not Fredholm on } L^{p(\cdot)}(\mathbb{T})\}
\]
of the operator $aP+Q$ is massive, that is, it has a nonzero plane measure.
\end{conjecture}

If $p:\mathbb{T}\to(1,\infty)$ is a measurable function such that
\begin{equation}\label{eq:bounded-exponent}
1<\operatornamewithlimits{ess\,inf}_{\tau\in\mathbb{T}}p(\tau),
\quad
\operatornamewithlimits{ess\,sup}_{\tau\in\mathbb{T}}p(\tau)<\infty
\end{equation}
and $S$ is bounded on $L^{p(\cdot)}(\mathbb{T})$, then from
\cite[Theorem~7.3]{Karlovich03} it follows that
\[
a(\mathbb{T}\setminus\{1\}) \cup\cS(0,1;0;\alpha,\beta)
\subset
\sigma_{{\rm ess}}(aP+Q)
\]
where $\alpha$ and $\beta$ are the lower and upper indices of the
following submultiplicative function
\[
Q(x):=\limsup_{R\to 0}
\frac{\|\chi_{\Delta(xR)}\|_{p(\cdot)}\|\chi_{\Delta(R)}\|_{q(\cdot)}}{|\Delta(R)|},
\quad
x\in(0,\infty),
\]
where $\Delta(R):=\Gamma(1,R)\setminus\Gamma(1,R/2)$.
Notice that if $p$ satisfies \eqref{eq:Dini-Lipschitz}, then $\alpha=\beta=1/p(1)$
by \cite[Lemma~5.8]{Karlovich03}.
If $0<\alpha<\beta<1$, then
the horn $\cS(0,1;0;\alpha,\beta)$ has a nonzero plane measure. Thus, to
confirm Conjecture~\ref{conjecture}, it is sufficient to construct an
exponent $p:\mathbb{T}\to(1,\infty)$ such that \eqref{eq:bounded-exponent}
holds, $S$ is bounded on $L^{p(\cdot)}(\mathbb{T})$, and $0<\alpha<\beta<1$.
\begin{remark}[On arbitrary Carleson curves]
Very recently the author \cite{Karlovich08} has obtained sufficient conditions
for the boundedness of the Cauchy singular integral operator $S$ on variable
Lebesgue spaces $L^{p(\cdot)}(\Gamma,w)$ with special weights $w(\tau)=|(\tau-t)^\gamma|$,
where $\gamma$ is a complex number, over arbitrary Carleson Jordan curves.
This allows him to obtain analogues of Theorems~\ref{th:Fredholm} and
\ref{th:main} for the case of nonweighted variable Lebesgue spaces
$L^{p(\cdot)}(\Gamma)$ over arbitrary Carleson curves (that is, without
condition \eqref{eq:spiralic}). As in the present paper, local spectra
of singular integral operators can be massive, but the reason for that is
different. In the present paper this effect is due to oscillation of weights,
but in the forthcoming paper \cite{Karlovich08} this effect is due to oscillation
of arbitrary Carleson curves.
\end{remark}


\begin{thebibliography}{99}
\bibitem{AK89}
V. D. Aslanov and Yu. I. Karlovich,
\textit{One-sided invertibility of functional operators in reflexive Orlicz spaces.}
Akad. Nauk Azerbaidzhan. SSR Dokl. \textbf{45} (1989), no. 11--12, 3--7 (in Russian).

\bibitem{BS88}
C. Bennett and R. Sharpley,
\textit{Interpolation of Operators.}
Academic Press, Boston, 1988.

\bibitem{BK95}
A. B\"ottcher and Yu. I. Karlovich,
\textit{Toeplitz
and singular integral operators on Carleson curves with
logarithmic whirl points.}
Integral Equations Operator Theory \textbf{22} (1995), 127--161.

\bibitem{BK97}
A. B\"ottcher and Yu. I. Karlovich,
\textit{Carleson Curves, Muckenhoupt Weights, and Toeplitz Operators.}
Birkh\"auser, Basel, 1997.

\bibitem{BK01}
A. B\"ottcher and Yu. I. Karlovich,
\textit{Cauchy's singular integral operator and its beautiful spectrum.}
In: ``Systems, approximation, singular integral operators,
and related topics'' (Bordeaux, 2000).
Operator Theory: Advances and Applications \textbf{129} (2001), 109--142.

\bibitem{BKR96}
A. B\"ottcher, Yu. I. Karlovich, and V. S. Rabinovich,
\textit{Emergence, persistence, and disappearance of logarithmic spirals in the
spectra of singular integral operators.}
Integral Equations Operator Theory \textbf{25} (1996), 406--444.

\bibitem{BKR00}
A. B\"ottcher, Yu. I. Karlovich, and V. S. Rabinovich,
\textit{The method of limit operators for one-dimensional singular integrals
with slowly oscillating data.}
J. Operator Theory \textbf{43} (2000), 171--198.

\bibitem{BS06}
A. B\"ottcher and B. Silbermann,
\textit{Analysis of Toeplitz Operators}. 2nd edition.
Springer-Verlag, Berlin, 2006.

\bibitem{Boyd71}
D. W. Boyd,
\textit{Indices for the Orlicz spaces.}
Pacific J. Math. \textbf{38} (1971), 315--323.

\bibitem{CG81}
K. P. Clancey and I. Gohberg,
\textit{Factorization of Matrix Functions and Singular Integral Operators.}
Operator Theory: Advances and Applications \textbf{3}.
Birkh\"auser, Basel, 1981.

\bibitem{David84}
G. David,
\textit{Oper\'ateurs int\'egraux singuliers sur certaines courbes du plan
complexe.}
Ann. Sci. \'Ecole Norm. Super. \textbf{17} (1984), 157--189.

\bibitem{FRS93}
T. Finck, S. Roch, and B. Silbermann,
\textit{Two projections theorems and symbol calculus for operators
with massive local spectra.}
Math. Nachr. \textbf{162} (1993), 167--185.

\bibitem{GK71}
I. Gohberg and N. Krupnik,
\textit{Singular integral operators with piecewise continuous coefficients
and their symbols.}
Math. USSR Izvestiya \textbf{5} (1971), 955--979.

\bibitem{GK92}
I. Gohberg and N. Krupnik,
\textit{One-Dimensional Linear Singular Integral Equations.}
Vols. 1 and 2.
Operator Theory: Advances and Applications \textbf{53--54}.
Birkh\"auser, Basel, 1992.

\bibitem{GK93}
I. Gohberg and N. Krupnik,
\textit{Extension theorems for Fredholm and invertibility symbols.}
Integral Equations Operator Theory \textbf{16} (1993), 514--529.

\bibitem{KK08}
E. Kapanadze and T. Kopaliani,
\textit{A note on maximal operator on $L^{p(t)}(\Omega)$ spaces.}
Georgian Math. J. \textbf{15} (2008), 307--316.

\bibitem{Karlovich96}
A. Yu. Karlovich,
\textit{Algebras of singular integral operators with piecewise
continuous coefficients on reflexive Orlicz spaces.}
Math. Nachr. \textbf{179} (1996), 187--222.

\bibitem{Karlovich98}
A. Yu. Karlovich,
\textit{Singular integral operators with piecewise continuous coefficients
in reflexive rearrangement-invariant spaces.}
Integral Equatations Operator Theory \textbf{32} (1998), 436--481.

\bibitem{K98-index}
A. Yu. Karlovich,
\textit{The index of singular integral operators in reflexive Orlicz spaces}.
Math. Notes \textbf{64} (1998), 330--341.

\bibitem{Karlovich02}
A. Yu. Karlovich,
\textit{Algebras of singular integral operators with $PC$ coefficients
in rearrangement-invari\-ant spaces with Muckenhoupt weights.}
J. Operator Theory \textbf{47} (2002), 303--323.

\bibitem{Karlovich03}
A. Yu. Karlovich,
\textit{Fredholmness of singular integral operators with piecewise continuous
coefficients on weighted Banach function spaces.}
J. Integr. Equat. Appl. \textbf{15} (2003), 263--320.

\bibitem{Karlovich05}
A. Yu. Karlovich,
\textit{Algebras of singular integral operators on Nakano spaces with
Khve\-delidze weights over Carleson curves with logarithmic whirl points.}
Ivzestija VUZov. Severo-Kavkazskii region. Estestvennie nauki, Special issue
``Pseudodifferential Equations and Some Problems of
Mathematical Physics", dedicated to 70th birthday of Prof. I. B. Simonenko,
Rostov-on-Don, 2005, 135--142.
Preprint is available at arXiv:math/0507312.

\bibitem{Karlovich06}
A. Yu. Karlovich,
\textit{Algebras of singular integral operators with piecewise
continuous coefficients on weighted Nakano spaces.}
In: ``The Extended Field of Operator Theory".
Operator Theory: Advances and Applications \textbf{171} (2006), 171--188.

\bibitem{Karlovich07}
A. Yu. Karlovich,
\textit{Semi-Fredholm singular integral operators with piecewise continuous
coefficients on weighted variable Lebesgue spaces are Fredholm.}
Operators and Matrices \textbf{1} (2007), 427--444.

\bibitem{Karlovich08}
A. Yu. Karlovich,
\textit{Singular integral operators on variable Lebesgue spaces over arbitrary Carleson
curves}. Preprint, arXiv:0810.3110, 2008.

\bibitem{K56}
B. V. Khvedelidze,
\textit{Linear discontinuous boundary problems in the theory of functions,
singular integral equations and some of their applications.}
Trudy Tbiliss. Mat. Inst. Razmadze \textbf{23} (1956), 3--158 (in Russian).

\bibitem{KPS05}
V. Kokilashvili, V. Paatashvili, and S. Samko,
\textit{Boundary value problems for analytic functions in the class of
Cauchy type integrals with density in $L^{p(\cdot)}(\Gamma)$.}
Boundary Value Problems \textbf{1} (2005), 43--71.

\bibitem{KPS06}
V. Kokilashvili, V. Paatashvili, and S. Samko,
\textit{Boundedness in Lebesgue spaces with variable exponent of the Cauchy
singular operator on Carleson curves.}
In: ``Modern Operator Theory and Applications. The Igor Borisovich Simonenko
Anniversary Volume".
Operator Theory: Advances and Applications \textbf{170} (2006), 167--186.

\bibitem{KSS07}
V. Kokilashvili, N. Samko, and S. Samko,
\textit{Singular operators in variable spaces $L^{p(\cdot)}(\Omega,\rho)$
with oscillating weights.}
Math. Nachr. \textbf{280} (2007), 1145--1156.

\bibitem{KS03}
V. Kokilashvili and S. Samko,
\textit{Singular integral equations in the Lebesgue spaces with variable exponent.}
Proc. A. Razmadze Math. Inst. \textbf{131} (2003), 61--78.

\bibitem{KR91}
O. Kov\'a{\v c}ik and J.~R\'akosn{\'\i}k,
\textit{On spaces $L\sp {p(x)}$ and $W\sp {k,p(x)}$.}
Czechoslovak Math. J. \textbf{41(116)} (1991), 592--618.

\bibitem{KPS82}
S. G. Krein, Ju. I. Petunin, and E. M. Semenov,
\textit{Interpolation of Linear Operators.}
AMS Translations of Mathematical Monographs \textbf{54},
Providence, RI, 1982.

\bibitem{Krupnik87}
N. Krupnik,
\textit{Banach Algebras with Symbol and Singular Integral Operators.}
Operator Theory: Advances and Applicastions \textbf{26}.
Birkh\"auser, Basel, 1987.

\bibitem{Lerner05}
A. K. Lerner,
\textit{Some remarks on the Hardy-Littlewood maximal function on variable
$L^p$ spaces.}
Math. Z. \textbf{251} (2005), 509--521.

\bibitem{LS87}
G. S. Litvinchuk and I. M. Spitkovsky,
\textit{Factorization of Measurable Matrix Functions.}
Operator Theory: Advances and Applications \textbf{25}.
Birkh\"auser, Basel, 1987.

\bibitem{M85}
L. Maligranda,
\textit{Indices and interpolation.}
Dissert. Math. \textbf{234} (1985), 1--49.

\bibitem{M89}
L. Maligranda,
\textit{Orlicz Spaces and Interpolation.}
Sem. Math. 5, Dep. Mat.,
Universidade Estadual de Campinas, Campinas SP, Brazil, 1989.

\bibitem{MO60}
W. Matuszewska and W. Orlicz,
\textit{On certain properties of $\varphi $-functions.}
Bull. Acad. Polon. Sci. Ser. Sci. Math. Astronom. Phys. \textbf{8} (1960), 439--443.
Reprinted in: W. Orlicz, \textit{Collected Papers}, PWN, Warszawa, 1988, 1112--1116.

\bibitem{MO65}
W. Matuszewska and W. Orlicz,
\textit{On some classes of functions with regard to their orders of growth.}
Studia Math. \textbf{26} (1965), 11--24.
Reprinted in: W. Orlicz, \textit{Collected Papers}, PWN, Warszawa, 1988, 1217--1230.

\bibitem{Musielak83}
J. Musielak,
\textit{Orlicz Spaces and Modular Spaces.}
Lecture Notes in Mathematics \textbf{1034}.
Springer-Verlag, Berlin, 1983.

\bibitem{Nakano50}
H. Nakano,
\textit{Modulared Semi-Ordered Linear Spaces.}
Maruzen Co., Ltd., Tokyo, 1950.

\bibitem{Rabinovich96}
V. S. Rabinovich,
\textit{Algebras of singular integral operators on composed contours with
nodes that are logarithmic whirl points.}
Izv. Math \textbf{60} (1996), 1261--1292.

\bibitem{RSS06}
V. Rabinovich, N. Samko, and S. Samko,
\textit{Local Fredholm spectra and Fredholm properties of singular integral
operators on Carleson curves acting on weighted H\"older spaces.}
Integral Equations Operator Theory \textbf{56} (2006), 257--283.

\bibitem{RS08}
V. Rabinovich and S. Samko,
\textit{Boundedness and Fredholmness of pseudodifferential operators in
variable exponent spaces.}
Integral Equations Operator Theory \textbf{60} (2008), 507--537.

\bibitem{RS90}
S. Roch and B. Silbermann,
\textit{Algebras of Convolution Operators and Their Image in the Calkin Algebra.}
Report R-Math-05/90, Karl Weierstrass Inst. F. Math., Berlin, 1990.

\bibitem{NSamko06}
N. Samko,
\textit{Singular integral operators in weighted spaces of continuous
functions with oscillating continuity moduli and oscillating weights.}
In: ``The Extended Field of Operator Theory".
Operator Theory: Advances and Applications \textbf{171} (2006), 323--347.

\bibitem{Simonenko64}
I. B. Simonenko,
\textit{The Riemann boundary value problem for $n$ pairs functions with
measurable coefficients and its application to the investigation of
singular integral operators in the spaces $L^p$ with weight.}
Izv. Akad. Nauk SSSR, Ser. Matem. \textbf{28} (1964), 277--306 (in Russian).

\bibitem{Simonenko65}
I. B. Simonenko,
\textit{A new general method of investigating linear operator equations
of singular integral equations type.}
Part I:  Izv. Akad. Nauk SSSR, Ser. Matem. \textbf{29} (1965), 567--586
(in Russian);
Part II: Izv. Akad. Nauk SSSR, Ser. Matem. \textbf{29} (1965), 757--782
(in Russian).

\bibitem{Simonenko68}
I. B. Simonenko,
\textit{Some general questions in the theory of the Riemann boundary value problem.}
Math. USSR Izvestiya \textbf{2} (1968), 1091--1099.

\bibitem{SM86}
I. B. Simonenko and Chin Ngok Min,
\textit{Local Method in the Theory of One-Dimensional Singular Integral Equations with
Piecewise Continuous Coefficients. Noetherity.}
Rostov-on-Don State Univ., Rostov-on-Don, 1986 (Russian).

\bibitem{Spitkovsky92}
I. M. Spitkovsky,
\textit{Singular integral operators with $PC$ symbols on the spaces with
general weights.}
J. Funct. Anal. \textbf{105} (1992), 129--143.

\bibitem{Widom60}
H. Widom,
\textit{Singular integral equations on $L^p$.}
Trans. Amer. Math. Soc. \textbf{97} (1960), 131--160.
\end{thebibliography}
\end{document}